\documentclass[a4paper, oneside,11pt,reqno]{amsart}
\usepackage{amsmath}
\usepackage{amssymb,amsbsy,amsmath,amsfonts,amssymb,amscd,amsthm}
\usepackage{mathtools}
\usepackage{stmaryrd}
\usepackage{mathrsfs}
\usepackage{bbm}
\usepackage{amssymb}
\usepackage{graphicx}
\usepackage{color}
\usepackage{enumerate}
\usepackage{a4wide}
\usepackage{comment}
\usepackage[final]{microtype}

\usepackage[dvipsnames]{xcolor}
\usepackage[colorlinks,linkcolor={red},citecolor={blue},urlcolor={purple}]{hyperref}

\def\N{{\mathbb N}}
\def\Z{{\mathbb Z}}

\def\R{{\mathbb R}}
\def\C{{\mathbb C}}

\def\E{{\mathbb E}}

\theoremstyle{plain}
\newtheorem{theorem}{Theorem}[section]
\theoremstyle{remark}
\newtheorem{remark}[theorem]{Remark}
\newtheorem{example}[theorem]{Example}
\theoremstyle{plain}

\newtheorem{lemma}[theorem]{Lemma}
\newtheorem{proposition}[theorem]{Proposition}
\newtheorem{definition}[theorem]{Definition}

\newtheorem{problem}{Problem}
\numberwithin{equation}{section}

\DeclareMathOperator{\pv}{p.v.}
\DeclareMathOperator{\Ran}{Ran}
\DeclarePairedDelimiter\abs{\lvert}{\rvert}

\DeclareMathOperator {\conv}{conv}
\DeclarePairedDelimiter\cbrace\{\}
\DeclarePairedDelimiter\ha()

\DeclarePairedDelimiter{\nrm}\lVert\rVert

\newcommand{\norm}[1]{\lVert #1 \rVert}

\newcommand{\nrms}[1]{\Bigl\|#1\Bigr\|}
\newcommand{\nrmb}[1]{\bigl\|#1\bigr\|}

\newcommand{\hab}[1]{\bigl(#1\bigr)}
\newcommand{\has}[1]{\Bigl(#1\Bigr)}
\newcommand{\cbraceb}[1]{\bigl\{#1\bigr\}}
\newcommand{\cbraces}[1]{\Bigl\{#1\Bigr\}}
\newcommand {\ud}{\,\mathrm{d}}

\newcommand{\dd}{\hspace{2pt}\mathrm{d}}
\DeclareMathOperator*{\esssup}{ess\,sup}
\newcommand{\one}{{{\bf 1}}}

\def\calL{{\mathcal L}}

\newcommand {\Schw}{\mathcal{S}}

\newcommand{\supp}{\mathrm{supp}\,}

\newcommand{\wt}{\widetilde}

\DeclareMathOperator {\UMD}{UMD}

\DeclareMathOperator {\LPR}{LPR}

 \numberwithin{equation}{section}

\newcommand{\F}{{\mathscr F}}

\newcommand{\wh}{\widehat}
\DeclareFontFamily{U}{mathx}{\hyphenchar\font45}
\DeclareFontShape{U}{mathx}{m}{n}{<5> <6> <7> <8> <9> <10> <10.95> <12> <14.4> <17.28> <20.74> <24.88> mathx10}{}
\DeclareSymbolFont{mathx}{U}{mathx}{m}{n}
\DeclareFontSubstitution{U}{mathx}{m}{n}
\DeclareMathAccent{\widecheck}{0}{mathx}{"71}

\usepackage{color}
\usepackage{xcolor}

\usepackage{caption}
\usepackage{subcaption}

\usepackage[shortlabels]{enumitem}

\allowdisplaybreaks

\begin{document}
	\title{Operator-valued Fourier multipliers of bounded $s$-variation}

\author[C. Deng]{Chenxi Deng}
\address[C. Deng]{Delft Institute of Applied Mathematics\\
	Delft University of Technology \\ P.O. Box 5031\\ 2600 GA Delft\\The
	Netherlands}
	\email{c.deng@tudelft.nl}

\author[E. Lorist]{Emiel Lorist}
\address[E. Lorist]{Delft Institute of Applied Mathematics\\
	Delft University of Technology \\ P.O. Box 5031\\ 2600 GA Delft\\The
	Netherlands}
\email{e.lorist@tudelft.nl}

\author[M.C. Veraar]{Mark Veraar}
\address[M.C. Veraar]{Delft Institute of Applied Mathematics\\
	Delft University of Technology \\ P.O. Box 5031\\ 2600 GA Delft\\The
	Netherlands}
\email{m.c.veraar@tudelft.nl}

\thanks{
The second author is supported by the Dutch Research Council (NWO) on the project ``The sparse revolution for stochastic partial differential equations'' with project number \href{https://doi.org/10.61686/ZGRMR99948}{VI.Veni.242.057}. The third author is supported by the VICI subsidy VI.C.212.027 of the Netherlands Organisation for Scientific Research (NWO)}

\begin{abstract}
	In this paper, we establish an operator-valued Fourier multiplier theorem in weighted Lebesgue spaces, Besov and Triebel--Lizorkin  spaces, assuming the multiplier has $\mathcal{R}$-bounded range and satisfies an $\ell^r$-summability condition on its bounded $s$-variation seminorms over dyadic intervals. The exponents $r$ and $s$ reflect the relationship between the geometric properties of the underlying Banach spaces (type and cotype) and the boundedness of Fourier multiplier operators. As our main tool we prove a weighted vector-valued variational Carleson inequality and deduce an estimate of Littlewood--Paley--Rubio de Francia type.
\end{abstract}
	


\maketitle

\section{Introduction}
Fourier multiplier operators are a central tool in harmonic analysis, with far-reaching applications to partial differential equations, function space theory, numerical analysis and analytic number theory. In this paper we are concerned with an operator-valued version of this theory. Before turning to this setting, we briefly recall some classical results from the scalar-valued case.

In the scalar-valued case, a Fourier multiplier operator $T_m$ for $m \colon \R\to\C$ is given by
$$T_mf: =\mathcal{F}^{-1}(m\cdot \mathcal{F}(f)), \quad f \in \mathcal{S}(\R),$$ where $\mathcal{F}$ denotes the Fourier transform. Sufficient conditions
for the boundedness of a Fourier multiplier operator $T_m$ on $L^p(\R)$ are provided by the classical Fourier multiplier theorems due to Marcinkiewicz \cite{Mar39}, Mihlin \cite{Mih56, Mih57} and H\"ormander \cite{Hor60}.
For the current paper the formulation of Marcinkiewicz' multiplier theorem is the most natural starting point. Here and below $\Delta$ is the dyadic interval partition of $\R$.
\begin{theorem}[Marcinkiewicz \cite{Mar39}]\label{5.8}
	Let $p \in (1,\infty)$ and let $m\colon \R \to \C$ such that for $J\in \Delta $, $m|_J\in V^1(J)$ and
	\begin{align*}
		\norm{m}_{\ell^\infty(V^1(\Delta))}&:=  \sup_{J \in \Delta} \norm{m|_J}_{V^1(J)}  <\infty,
	\end{align*}
	where $V^1(J)$ is the space of functions of bounded variation. Then $T_m$ is bounded on $L^p(\R)$.
\end{theorem}
In \cite{coifman}, Coifman, Rubio de Francia and Semmes relaxed the assumption of bounded variation in Theorem \ref{5.8} to bounded $s$-variation  for some $s>1$ (see Section \ref{sec:RVC} below). Bounded $s$-variation is weaker for larger $s$ and is implied by $\frac1s$-H\"older smoothness. Coifman et al.\ proved their result for the boundedness of $T_m$ on $L^p(\R)$ for $m \in \ell^\infty(V^s(\Delta))$ in three steps:
\begin{enumerate}[(i)]
\item First, they considered the case $p \in [2,\infty)$ and $s\in [1, 2)$, using the extension of the Littlewood--Paley inequality of Rubio de Francia to arbitrary intervals of \cite{Rubio}:
\begin{align}\label{eq:Rubioscalar}
\Big\| \Big(\sum_{I\in \mathcal{I}} |S_I f|^2 \Big)^{1/2} \Big\|_{L^p(\R)} \leq C_p \|f\|_{L^p(\R)},
\end{align}
where $p\in [2, \infty)$, $\mathcal{I}$ is a family of disjoint intervals in $\R$ and $S_I = T_{m}$ for $m:=\one_I$.
\item Secondly, the case $p \in (1,2)$ and $s\in [1, 2)$ follows by duality.
\item Finally, note that the map $(m,f) \mapsto T_mf$ is bounded from  $ L^\infty(\R) \times L^2(\R) $ to $L^2(\R)$ by the Plancherel theorem. Hence, using bilinear interpolation they concluded that $T_m$ is bounded on $L^p(\R)$ if $m \in \ell^\infty(V^s(\Delta))$ with $\frac{1}{s} > \abs{\frac{1}{p}-\frac{1}{2}}$ for $p,s\in (1, \infty)$.
\end{enumerate}
Related multiplier results have been studied in  \cite{BG98, Krol, Ku80}.

\subsection{Fourier multipliers in the vector-valued setting}\label{ss:introFouriervector}
Now let $X$ and $Y$ be Banach spaces.
An operator-valued Fourier multiplier operator $T_m$ is of the form
\begin{equation*}
	T_{m}f =\mathcal{F}^{-1}(m\cdot\mathcal{F}(f)),\quad f \in \mathcal{S}(\R;X),
\end{equation*} where $m \colon \R \to \mathcal{L}(X,Y)$, the space $\mathcal{S}(\R;X)$ denotes the $X$-valued Schwartz functions and $\mathcal{F}$ denotes the Fourier transform.

The classical multiplier theorems due to Marcinkiewicz and Mihlin were extended to $\UMD$ spaces in two stages.  The case $X = Y$ with scalar-valued multipliers $m$ was first addressed in the 1980s by McConnell \cite{McC84} and Bourgain \cite{Bourgain714}, and extended by Zimmermann \cite{Zim} (see also \cite{Tuomas}).
Subsequently, in the early 2000s, the more general setting with $X \neq Y$ and operator-valued multipliers $m$ was developed and was motivated by applications to evolution equations such as maximal $L^p$-regularity and stability of semigroups (see \cite{WeisStability}).
An operator-valued analogue of the Mihlin multiplier theorem was proven by Weis \cite{Weis01} and applied to maximal $L^p$-regularity, and shortly after an analogue of the Marcinkiewicz multiplier theorem by \v{S}trkalj and Weis \cite{SW07}.  These theorems require a so-called $\mathcal{R}$-boundedness condition (see \cite[Chapter 8]{HNVW2}) on the range of $m$,  a strengthening of uniform boundedness that was shown to be necessary by Cl\'ement and Pr\"uss \cite{Clement}. The reader is referred to the Notes of Chapter 8 in \cite{HNVW2} for further references.

This naturally leads to the question whether the aforementioned result of Coifman et al. \cite{coifman} admits an extension to the operator-valued setting. This programme was initiated by Hyt\"onen and Potapov \cite{hp}, where the three main steps mentioned above are followed.
\begin{enumerate}[(i)]
  \item The first step in the operator-valued setting builds on two hypotheses: the range of $m$ is $\mathcal{R}$-bounded and the availability of an $X$-valued version of the Littlewood--Paley--Rubio de Francia estimate \eqref{eq:Rubioscalar}, introduced by Berkson, Gillespie and Torrea \cite{BGT03} and called the $\LPR_p$-property of $X$. This estimate was further studied in \cite{ALV,GT04,HTY,PSX12}.
       At present, the $\LPR_p$-property is only known to hold for Banach \emph{function} spaces satisfying additional geometric requirements, such as the requirement that the $2$-concavification $X^2$ is UMD (see \cite{ALV, PSX12}). It is an open problem whether, e.g., the Schatten class $S^t$ with $t\in [2,\infty)$ has the $\LPR_p$-property.
  \item In the second step one can still employ a duality argument.
  \item In the third step, Hyt\"onen and Potapov \cite{hp} additionally require that $X$ is a complex interpolation space between a Banach space $X_0$ with the $\LPR_p$-property and a Hilbert space $H$. Furthermore, the $\mathcal{R}$-boundedness condition has to be replaced by a more intricate condition: the existence of an $\mathcal{R}$-bounded family $\mathcal{T}\subseteq \calL(X_0\cap H)$ such that $\Ran(m)\subseteq\mathrm{span}\ha{\mathcal{T}}$ and $m$ has bounded $s$-variation in the Minkowski functional with respect to $\mathcal{T}$. This requirement can be quite delicate to check.
\end{enumerate}
A further generalization of this approach for Banach function spaces  was obtained in \cite{alv2}.

\subsection{Special case of our main result}
The goal of this paper is to prove a weighted operator-valued Fourier multiplier theorem for symbols $m$ of bounded $s$-variation with $s>1$, where $s$ is related to the geometry of the underlying Banach spaces. We aim for assumptions on the underlying Banach spaces that can be verified beyond the class of Banach function spaces and, at the same time, avoid the intricate operator-theoretic conditions arising from the interpolation argument sketched above.
 To do so, we will introduce a different type of Rubio de Francia property.

A simplified version of our main result reads as follows.
A space $X$ is called $\theta$-intermediate $\UMD$ if it can be written as $X = [X_0, H]_{\theta}$ for some $\theta\in (0,1]$. We will use the class $A_p$ of Muckenhoupt weights for $p\in (1, \infty)$.

\begin{theorem}\label{thm:mainintro}
Let $\theta\in (0,1]$, $t\in (1, 2]$, $q\in [2, \infty)$ and set $\frac1r := \frac1t-\frac1q$. Let $X$ be a $\theta$-intermediate $\UMD$ Banach space with cotype $q$ and $Y$ be a $\UMD$ Banach space with type $t$. Let $s\in[1,\frac{2}{2-\theta})$ and suppose that $m \colon \R \to \mathcal{L}(X,Y)$ has $\mathcal{R}$-bounded range and
	$$
	\|m\|_{\ell^{r}(\dot{V}^{s}(\Delta;\mathcal{L}(X,Y)))} := \has{\sum_{J \in \Delta} [m|_J]_{V^s(J;\mathcal{L}(X,Y))}^r}^{\frac1r}
	<\infty.
	$$	
Then $T_m$ is bounded from $L^p(\R,w;X)$ to $L^p(\R,w;Y)$ for all $p\in (s,\infty)$ and $w\in A_{p/s}$.

Moreover, if, additionally, $Y$ is a $\theta$-intermediate $\UMD$ Banach space, then  $T_m$ is bounded from $L^p(\R;X)$ to $L^p(\R;Y)$
 for all $p\in (1, \infty)$.
\end{theorem}

Theorem \ref{thm:mainintro} is established in the main text as Theorem \ref{thm:mainthm}. One can assume without loss of generality that $r>s$. Moreover, the result is applicable to $m(\xi) = \rm{sign}(\xi)$, so the UMD property of $X$ and $Y$ is necessary in general (see \cite[Chapter 5]{HNVW1}).

Our result is the first of its kind on general $\theta$-intermediate UMD spaces. As discussed at the end of Subsection \ref{ss:introFouriervector}, previous results in the operator-valued setting either only apply to Banach function spaces or include assumptions on $m$ in the interpolation endpoints of $X=[X_0,H]_\theta$.

Compared to the operator-valued analogue of Theorem \ref{5.8} (see \cite{HHN,SW07} and \cite[Theorem 8.3.9]{HNVW2}), we note that to weaken the $V^1$-boundedness to $V^s$-boundedness we need to impose a decay condition at zero and infinity in terms of $\ell^r$-summability unless $r=\infty$. On the other hand, we do not need to use any Minkowski functionals related to an  $\mathcal{R}$-bounded family $\mathcal{T}$ containing the range of $m$ in its linear span.

We do not know if the decay condition in Theorem \ref{thm:mainintro} is necessary. It can, for example, be avoided if
 $X$ has cotype $2$ and $Y$ has type $2$. It can also be omitted if the $L^p$-spaces are replaced by homogeneous Besov spaces $\dot{B}^{\alpha}_{p,q}$, in which case the condition on $m$ takes the simpler form
\[\sup_{J\in \Delta}\|m\|_{V^{s}(J;\mathcal{L}(X,Y))}<\infty.\]
Note that in this case type and cotype no longer play a role in the assumptions on $m$. A similar result holds for the inhomogeneous Besov spaces (see Theorem \ref{thm:Besov}) and in the Triebel--Lizorkin scale (see Theorem \ref{thm:TL}).

\bigskip

Theorem~\ref{thm:mainintro} yields various Fourier multipliers that are not covered by the Marcinkiewicz or Mihlin multiplier theorems. Indeed, for $s>1$, the space $V^s$ is substantially larger than $V^1$. There is also an { $L^1$-approach} to construct multipliers outside the Marcinkiewicz–Mihlin class, based on the $L^1$-integrability of $\F^{-1}m$ or $\F^{-1}(\one_J m)$ in the strong operator topology. The $L^1$-approach plays a central role in \cite{GW03Besov,GW03}, where Fourier multiplier theorems on Besov and Lebesgue–Bochner spaces are obtained under the assumption that $X$ and $Y$ have Fourier type. One of the main building blocks there is a Steklin type multiplier theorem:
 $$m\in B^{1/p}_{p,1}(\R;\calL(X,Y)) \rightarrow \mathcal{F}^{-1}mx\in L^1(\R;Y) \text{ and } \mathcal{F}^{-1}my^* \in L^1(\R;X^*), \quad x\in X,y^*\in Y^*.$$
  Since $B^{1/p}_{p,1}(\R;\calL(X,Y))$ embeds into the space of continuous functions (\cite[Corollary 14.4.27]{HNVW3}), the multipliers satisfying the assumptions of Theorem~\ref{thm:mainintro}, being not necessarily continuous, do not fit into this $L^1$-approach.

On the other hand, in the next example we show that Theorem~\ref{thm:mainintro} does not fully cover all multipliers satisfying the assumptions in the operator-valued Mihlin or Marcinkiewicz multiplier theorems. This is  due to the decay condition at infinity in case $r<\infty$ in Theorem \ref{thm:mainintro}.
\begin{example}\label{ex:multMR}
Let $X = \ell^p$ with $p\in (1, \infty)$ and let $(e_n)_{n\in \N}$ be the canonical basis of $ \ell^p$. Let $A$ be a multiplication operator on $\ell^p$ given by $A e_n = 2^n  e_n$ for $n\in \N$ with its natural domain $D(A) = \{x\in \ell^p: A x\in \ell^p\}$. Define $$m(\xi) = A (i\xi+ A)^{-1}, \quad \xi \in \R.$$ Then we have for $n \in \N$,
\begin{align*}
\|m(2^{n+1}) - m(2^n)\|_{\calL(\ell^p)} &\geq \|m(2^{n+1}) e_n - m(2^n) e_n\|_{\ell^p}
\\ & = |2^n (i 2^{n+1} + 2^n)^{-1} - 2^n (i 2^{n} + 2^n)^{-1}|
\\ & = |(2 i + 1)^{-1} - (i  + 1)^{-1}| = \frac{1}{\sqrt{10}}.
\end{align*}
Since $m$ is continuous, we find that $[m]_{V^s(J;\calL(\ell^p))}\geq \frac{1}{\sqrt{10}}$ for every dyadic interval $J\in \Delta$. Therefore, $m \notin \ell^{r}(\dot{V}^{s}(\Delta;\mathcal{L}(\ell^p)))$ for $r<\infty$. However, it is known that $m$ satisfies the operator-valued Mihlin or Marcinkiewicz multiplier theorem conditions, see, for instance, \cite[Proposition 6.1.3]{KLW}.
\end{example}

We conclude this subsection with an example that does not seem to be covered by any existing result. The closest statement we are aware of is \cite[Example 5.19]{alv2}, but there the hypotheses on $m$ are of a different nature. In the scalar-valued case, the result from \cite{alv2} is considerably stronger.
\begin{example}
Fix $t\in (1, 2)$ and let $X = S^t$ be the Schatten class on $\ell^2$. Then $X$ has type $t$ and cotype $2$ (see \cite[Proposition 7.1.11]{HNVW2}), and it is $\theta$-intermediate UMD for all $\theta\in (0,2-\frac{2}{t})$ (see \cite[Propositions 5.4.2 and D.3.1]{HNVW1}). Set $\frac{1}{r} = \frac1t - \frac12$ and choose $s\in (1, t)$. Theorem \ref{thm:mainintro} then yields that if $m\in \ell^{r}(\dot{V}^{s}(\Delta;\mathcal{L}(S^t)))$ and $m$ has $\mathcal{R}$-bounded range, then $T_m$ is bounded on $L^p(\R;S^t)$ for all $p\in (1,\infty)$.
An analogous statement holds for $t\in (2,\infty)$ and follows either by a similar argument or by duality.
\end{example}

\subsection{Methodology}

To prove Theorem \ref{thm:mainintro}, we establish a new vector-valued estimate of Littlewood--Paley--Rubio de Francia type, which in some sense is weaker and thus easier to check than the $\LPR_p$-property, as will be explained below Proposition \ref{prop:RubioBGT}.
\begin{theorem}\label{thm:RubioX}
Let $\theta\in (0,1]$, $q\in (2/\theta, \infty)$, $p\in (q', \infty)$ and let $w\in A_{p/q'}$. Let $X$ be a $\theta$-intermediate $\UMD$ Banach space and let $\mathcal{I}$ be a family of disjoint intervals of $\R$. Then there exists an increasing function $\phi_{X,p,q}:[1, \infty)\to [1, \infty)$ such that
\begin{align}\label{eq:RubioX}
  \Big\|\big(\sum_{I\in \mathcal{I}}\|S_{I}f\|^{q}_X\big)^{\frac{1}{q}}\Big\|_{L^p(\R,w)}\leq \phi_{X,p,q}([w]_{A_{p/q'}}) \|f\|_{L^p(\R,w;X)}.
\end{align}
\end{theorem}

The excluded  endpoint $q=2$ when $\theta=1$ (i.e.\ $X$ is a Hilbert space) can be found in \cite[Proposition 6.6]{ALV}, which is a weighted vector-valued extension of \eqref{eq:Rubioscalar}. Theorem \ref{thm:RubioX} in the scalar case with  $w=1$ can be found in \cite{Rubio} and weights were added in \cite{Krol}.
The limiting case $p=q'$ was shown to be false in \cite{CT}.

 We will establish Theorem \ref{thm:RubioX} as a special case of the weighted boundedness of the  variational Carleson operator, see Theorem \ref{thm:varcarleson}. Compared to \eqref{eq:RubioX}, the variational Carleson operator has an additional supremum over all $\mathcal{I}$ inside the $L^p$-norm. The analogue for $q=\infty$ is usually called Carleson's inequality, and is the main tool to establish almost everywhere convergence of Fourier series \cite{Carleson, Hunt}.
The (weighted) boundedness of the variational Carleson operator in the scalar-valued case was established in  \cite{DiDoUr, DoLac, ObSeTaThWr}.
The unweighted vector-valued boundedness of the Carleson operator (i.e.\ Theorem~\ref{thm:RubioX} with 
	$q=\infty$) was established in \cite{HL} for intermediate UMD spaces. The unweighted vector-valued boundedness of the variational Carleson operator was studied in \cite{amenta2022}, which in particular implies \eqref{eq:RubioX} in the case  $w=1$.

We will establish the weighted boundedness of the variational Carleson operator by extending \cite{HL} to the weighted setting through \cite{Lorist21}, and afterwards interpolating the result with weighted bounds for the variational Carleson operator on $L^p(\R;H)$ for Hilbert spaces $H$ from \cite{ALV, DiDoUr} (see also \cite{LN22}).
We note that the following variant of \eqref{eq:RubioX} was developed in \cite[Section 2.2]{DengLorVer} in the periodic setting:
\begin{align}\label{eq:decompX}
  \Big(\sum_{I\in \mathcal{I}}\|S_{I}f\|^{q}_{L^p(\R;X)}\Big)^{\frac{1}{q}}\lesssim \|f\|_{L^p(\R;X)}.
\end{align}
By Minkowski's inequality, either \eqref{eq:RubioX} or \eqref{eq:decompX} is stronger, depending on the relative size of $p$ and $q$.

As a consequence of Theorem \ref{thm:RubioX}, we will show that for $m\in V^s(\R;\calL(X,Y))$ with $s\in[1,\frac{2}{2-\theta})$, one has that $T_m$ is bounded from $L^p(\R,w;X)$ to $L^p(\R,w;Y)$ for all $p\in (s,\infty)$ and $w\in A_{p/s}$ in Theorem \ref{thm:boundedsupp}. In the proof of this result, $\mathcal{R}$-boundedness of the range of $m$ does not play any role. This is rather surprising, since it is known that the  $\mathcal{R}$-boundedness is necessary for the boundedness of $T_m$ (see \cite{Clement} and \cite[Theorem 5.3.15]{HNVW1}). In Theorems \ref{thm:mainintro} and \ref{thm:mainthm}, $\mathcal{R}$-boundedness enters in an indirect way by a decomposition of $m$, where one part in the decomposition is constant on dyadic intervals.

 In Proposition \ref{R-bdd of R}, we will show that in the setting of Theorem \ref{thm:mainintro}, one actually has that any $m\in \ell^{r}(R^{r}(\mathcal{J};\mathcal{L}(X,Y)))$ has $\mathcal{R}$-bounded range, where $\mathcal{J}$ is a collection of disjoint intervals in $\R$. This is a new result on $\mathcal{R}$-boundedness which complements the work of \cite{HytVer} and came as a surprise to the authors. Here $R^r$ is the space used in \cite{coifman} defined via certain atoms, for which one has $R^s\hookrightarrow V^{s}\hookrightarrow R^{r}$ with $s<r$. Motivated by the above, we will prove a general embedding result for Besov spaces and $V^r$ and $R^r$ in Lemma \ref{lem:BesovV}.

\subsection{Organization} In Section \ref{sec:m-prelim} we present some preliminaries. Moreover, we introduce $V^s$ and $R^s$, and prove some new results for them. Then in Section \ref{sec:m-R} we introduce $\mathcal{R}$-boundedness and give a new sufficient condition for it in terms of the $V^s$ and $R^r$-regularity of a function. In Section \ref{sec:m-carleson} we prepare for our main result by establishing a weighted version of the vector-valued variational Carleson estimate, which, in particular, implies a weighted version of Theorem \ref{thm:RubioX}. Afterwards, we establish our main result in Section \ref{sec:m-main}.
At the end of the paper we list several open problems.

\section{Preliminaries}\label{sec:m-prelim}
\subsection{Notation}
The space of bounded linear operators between complex Banach spaces \(X\) and \(Y\) is denoted by \(\mathcal{L}(X, Y)\), and \(\mathcal{L}(X) := \mathcal{L}(X, X)\). We write $X\hookrightarrow Y$ if $X$ embeds in $Y$ continuously. The space of \(X\)-valued tempered distributions is denoted by \(\mathcal{S}'(\R; X)\).
The Fourier transform of $f \in \mathcal{S}'(\mathbb{R}; X)$ is denoted by $\mathcal{F}f$ or $\widehat{f}$. If $f \in L^1(\mathbb{R}; X)$, then
$$
\widehat{f}(\xi) := \int_{\mathbb{R}} e^{-2\pi i \xi t} f(t)\, \mathrm{d}t, \quad \xi \in \mathbb{R}.
$$
Let $\Delta$ denote the dyadic partition of $\mathbb{R}$, i.e.
$ \Delta := \bigcup_{k \in \mathbb{Z}} \pm [2^k, 2^{k+1}).$ The restriction of a function $f:\R\to X$ to $J\subseteq \R$ is defined by $f|_J$. The range of $f$ is denoted by \(\Ran(f)\).

Constants depending on parameters \(a, b, \ldots\) are denoted by \(C_{a, b, \ldots}\), and their values may vary from line to line. We will often suppress the constant $C_{a, b, \ldots}$ by writing $\lesssim_{a,b,\ldots}$. We write $\phi_{a,b,\ldots}\colon [1,\infty)\to [1,\infty)$ to denote a non-decreasing function
which depends only on the parameters $a, b, \ldots$ and which may also
change from line to line.

\subsection{Weights}
For $p\in(1,\infty)$, we define the \emph{ Muckenhoupt $A_p$-class} as the class of all locally integrable  $w \colon \R \to (0,\infty)$ such that
$$[w]_{A_p}:=\underset{J}{\sup}\Bigl(\frac{1}{|J|}\int_{J} w(x)\ud x\Bigr)\Bigl(\frac{1}{|J|}\int_{J} w(x)^{1-p'}\ud x\Bigr)^{p-1}<\infty,$$
where the supremum is taken over all bounded intervals $J\subseteq \mathbb{R}$.  The Muckenhoupt classes are increasing in $p$, i.e. if $p_0<p_1$, then
$A_{p_0}\subseteq A_{p_1}$  with $[w]_{A_{p_1}} \leq [w]_{A_{p_0}}$.
Let $A_{\infty} = \bigcup_{p>1} A_p$.

We have the following self-improvement lemma (see \cite[Theorem 1.2]{HPR12} and \cite[Corollary 7.2.6]{Grafakos}).
\begin{lemma}\label{self-improvement}
Let $p\in(1,\infty)$,  then for any $w \in A_{p}$ we have
$$
[w]_{A_{p-\varepsilon}} \lesssim_{p}[w]_{A_p}, \quad 0\leq \varepsilon \leq \frac{p-1}{1+C[w]_{A_p}^{(p-1)^{-1}}}.
$$
\end{lemma}
For a Banach space $X$,  $p\in(1,\infty)$ and $w\in A_p$ we define $L^{p}(\R,w;X)$ as the space of all strongly measurable $f \colon \R \to X$ such that
\[
\|f\|_{L^{p}(\R,w;X)}:=\Big(\int_{\R}\|f(x)\|_{X}^{p}w(x)\ud x\Big)^{1/p}<\infty.
\]
Recall from \cite[Lemma 3.3]{FHL} that $\mathcal{S}(\R;X)$ is dense in $L^p(\R,w; X)$.

\subsection{UMD spaces}
For an introduction to $\UMD$ Banach spaces the reader is referred to \cite[Chapters 4 and 5]{HNVW1} and \cite{Pi16}. We briefly recall some basics.
Let $X$ be a Banach space. We denote by  $\mathcal{H}$ the Hilbert transform, for $f\in \mathcal{S}(\R;X)$ given by the principal value integral
\begin{align}\label{Hilbert-transform}
\mathcal{H}f(x):=\pv \frac{1}{\pi}\int_{\R}\frac{f(y)}{x-y}\dd y,\quad x\in\R.
\end{align}
A Banach space is said to have the {\em $\UMD$ property} if for some (equivalently all) $p \in (1,\infty)$  one has that the Hilbert transform extends to a bounded operator on $L^p(\R;X)$. By the seminal works of Burkholder \cite{Burkholder} and Bourgain \cite{Bourgainmartingale}, this is equivalent to the original definition of the $\UMD$ property in terms of the {U}nconditionality of {M}artingale {D}ifferences.

All reflexive  Lebesgue, Sobolev, Besov and Triebel--Lizorkin spaces are  $\UMD$ Banach spaces. The $\UMD$ property implies reflexivity, so $L^1(\R)$ and $L^{\infty}(\R)$ are not $\UMD$ Banach spaces.

Recall that for an interval $I\subseteq \R$, $S_I$ is defined as the Fourier multiplier operator with symbol $m = \one_I$. The vector-valued analogue of the Littlewood–Paley inequality for $\UMD$ spaces was obtained in \cite{Bourgain714} (see also \cite[Chapter 5]{HNVW1}). The following weighted extension can be found in \cite[Theorem 3.4]{FHL}.
\begin{proposition}\label{prop:LPw}
Let $X$ be a $\UMD$ Banach space, $(\varepsilon_J)_{J\in \Delta}$ be a Rademacher sequence, $p\in (1,\infty)$ and $w\in A_p$. Then we have for all $f \in L^p(\R,w;X)$,
\begin{align*}
\mathbb{E}\big\|\sum_{J\in \Delta}\varepsilon_J S_J f\big\|_{L^p(\mathbb{R},w;X)}&\lesssim_{X,p}[w]^{2\max\{1,\frac{1}{p-1}\}}_{A_p}\left\|f\right\|_{L^p(\mathbb{R},w;X)},\\
\left\|f\right\|_{L^p(\mathbb{R},w;X)}&\lesssim_{X,p}[w]^{2\max\{1,\frac{1}{p-1}\}}_{A_p} \mathbb{E}\big\|\sum_{J\in \Delta}\varepsilon_J S_J f\big\|_{L^p(\mathbb{R},w;X)}.
\end{align*}
\end{proposition}

In the next definition we will use the complex interpolation method $[X,Y]_\theta$, for which we refer to \cite{BL}.
\begin{definition}
Let $\theta\in(0,1]$.	A Banach space $X$ is called a  {\rm $\theta$-intermediate $\UMD$  Banach space} if it can be written as $X=[Z,H]_{\theta}$ for a $\UMD$ Banach space $Z$ and a Hilbert space $H$.
\end{definition}
If $\theta =1$ we have $X=H$, i.e. $X$ is a Hilbert space. If $X$ is $\theta$-intermediate $\UMD$, then its dual is also $\theta$-intermediate $\UMD$ (see \cite[Corollary 4.5.2]{BL}). Since the $\UMD$ property is stable under complex interpolation, any intermediate $\UMD$ Banach space is a $\UMD$ Banach space. Conversely, all $\UMD$ Banach \emph{function} spaces are intermediate $\UMD$ Banach function spaces by a result of Rubio de Francia \cite{Rubio86}. For general Banach spaces this is an open problem.

\subsection{Type and cotype}
Let $(\varepsilon_n)_{n\ge 1}$ be the Rademacher sequence.
A Banach space $X$ is said to have {\it type $t \in[1, 2]$} if there exists a constant $C\ge1$ such that for all finite sequences $(x_n)_{n=1}^N$ in $X$, we have
$$\Bigl(\mathbb{E}\nrmb{\sum_{n=1}^N\varepsilon_n x_n}^2\Bigr)^{\frac{1}{2}}\leq C\,\Bigl(\sum_{n=1}^N\|x_n\|^t\Bigr)^{\frac{1}{t}}.$$
A Banach space $X$ is said to have {\it cotype ${q} \in[2,\infty]$} if there exists a constant $C\ge1$ such that for all finite sequences $(x_n)_{n=1}^N$ in $X$, we have
$$\Bigl(\sum_{n=1}^N\|x_n\|^{{q}}\Bigr)^{\frac{1}{{q}}}\leq C\,\Bigl(\mathbb{E}\nrmb{\sum_{n=1}^N\varepsilon_n x_n}^{{2}}\Bigr)^{\frac{1}{2}},$$
with the usual modification for ${q}=\infty$.
Every Banach space has type 1 and cotype $\infty$. The space $X$ has type 2 and cotype 2 if and only if $X$ is isomorphic to a Hilbert space (\cite[Proposition 3.1]{Kwapien}).
For $p\in [1, \infty)$ the (noncommutative) $L^p$-spaces are known to have type $p\wedge 2$ and cotype $p\vee 2$, see \cite{Fack87,HNVW2}.

Finally, we note that for a $\sigma$-finite measure space $(S,\mu)$, if $X$ has type $t$ and $p\in [1, \infty)$, then $L^p(S;X)$ has type $t\wedge p$. Moreover, if $X$ has cotype $q$ and $p\in [1, \infty]$, then  $L^p(S;X)$ has cotype $p\vee q$. In particular, the space $L^2(S;X)$ has the same type and cotype as $X$ (see \cite[Proposition 7.1.4]{HNVW2}).

\subsection{The function spaces \texorpdfstring{$R^s$, $V^s$ and $C^\alpha$}{Rs, Vs and Ca}}\label{sec:RVC}
In this subsection we introduce the spaces of bounded $s$-variation $V^s$ for $s \in [1, \infty)$, which will play a central role in the main result of the paper.

Let $X$ be a Banach space and $J\subseteq \mathbb{R}$, $s\in[1,\infty)$. For $f\colon \R \to X$ define
\begin{equation}\label{7.30}
\left[f\right]_{V^s(J;X)}  := \sup_{t_0<\cdots<t_{N}} \left(\sum_{i=0}^{N-1} \left\|f(t_{i+1}) - f(t_i)\right\|_X^s\right)^{1/s}.
\end{equation}
We write $f\in \dot{V}^{s}(J;X)$ if \eqref{7.30} is finite.
We say that $f$ has \emph{bounded $s$-variation for $s\in [1,\infty)$}, denoted by $f \in V^s(J;X)$, if
\begin{equation}\label{Vnorm}
\left\|{f}\right\|_{V^s(J;X)} := \left\|f\right\|_{L^\infty(J;X)} +  \left[f\right]_{V^s(J;X)} <\infty.
\end{equation}
Furthermore, we use the convention $V^\infty(J;X):= L^\infty(J;X)$. Note that $V^s(J;X)$ is a Banach space.

Given a collection of disjoint sets $\mathcal{J}$ in $\R$ and $r \in [1,\infty]$, let $\ell^r(V^s(\mathcal{J};X))$ and $\ell^r(\dot{V}^s(\mathcal{J};X))$ denote the space of all $f \in L^\infty(\R;X)$  such that
\begin{equation*}
\norm{f}_{\ell^r(V^s(\mathcal{J};X))} := \has{\sum_{J \in \mathcal{J}} \norm{f|_J}_{V^s(J;X)}^r}^{1/r} < \infty,
\end{equation*}
and \begin{equation*}
\norm{f}_{\ell^r(\dot{V}^s(\mathcal{J};X))} := \has{\sum_{J \in \mathcal{J}} [{f|_J}]_{V^s(J;X)}^r}^{1/r} < \infty,
\end{equation*}
with the usual modifications for $r=\infty$.

Now suppose that $J\subseteq \mathbb{R}$ is an \emph{interval}. We say that a function $a\colon J\to X$ is an \emph{$R^s(J;X)$-atom}, written as $a \in R^s_{\mathrm{at}}(J;X)$, if there exists a set $\mathcal{I}$ of mutually disjoint subintervals of $J$ and a set of vectors $(c_I)_{I \in \mathcal{I}}$ in $X$ such that
\begin{equation*}
a = \sum_{I \in \mathcal{I}} c_I \one_I \quad \text{and} \quad \has{ \sum_{I \in \mathcal{I}}  \norm{c_I}_X^s }^{1/s} \leq 1.
\end{equation*}	
Define $R^s(J;X)$ by    \begin{equation*}
R^s(J;X) := \cbraces{ f\in L^\infty(J;X): f = \sum_{k=1}^\infty \lambda_k a_k,  ( \lambda_k)_{k\geq 1} \in \ell^1, (a_k)_{k\geq 1} \subseteq R_{{\mathrm{at}}}^s(J;X)},
\end{equation*}
with norm
\begin{equation*}
\left\|{f}\right\|_{R^s(J;X)} := \inf\cbraceb{ \left\|( \lambda_k)_{k\geq 1}\right\|_{\ell^1} : \text{$f = \sum_{k=1}^\infty  \lambda_k a_k$ as above} }.
\end{equation*}
Again we let $R^\infty(J;X) = L^\infty(J;X)$. Note that $R^s(J;X)$ is a Banach space.

For a collection of  disjoint intervals $\mathcal{J}$ in $\R$ and $r \in [1,\infty]$, the space $\ell^r(R^s(\mathcal{J};X))$ consists of all $f \in L^\infty(\R;X)$ such that
\begin{equation*}
\left\|{f}\right\|_{\ell^r(R^s(\mathcal{J};X))} := \has{\sum_{J \in \mathcal{J}} \norm{f|_J}_{R^s(J;X)}^r}^{1/r} < \infty,
\end{equation*}
with the usual modification for $r=\infty$.

For $\alpha \in(0,1]$ we define the space of \emph{  $\alpha$-H{\"o}lder continuous functions} $C^{\alpha}(J;X)$ as the space of all $f\colon J \to X$ with
$$\|f\|_{C^{\alpha}(J;X)}:=\|f\|_{L^\infty(J;X)}+[f]_{C^{\alpha}(J;X)},$$
where $$[f]_{C^{\alpha}(J;X)}:=\sup_{t_1,t_2\in J, t_1\neq t_2}\frac{\|f(t_1)-f(t_2)\|_X}{|t_1-t_2|^{\alpha}}.$$
It is standard to check that $C^{\alpha}(J;X)$ is a Banach space.

The following embedding results can be found in  \cite[Lemma 4.3]{alv2} and \cite[Lemma 2]{coifman}.
\begin{lemma}\label{lemma:5.17}
Let $X$ be a Banach space, let $J\subseteq \mathbb{R}$ be an interval, and $s\in[1,\infty)$. Then
\begin{enumerate}[(i)]
	\item $R^{s}(J;X) \hookrightarrow V^{s}(J;X) $ and for all $f\in R^{s}(J;X)$,
	\begin{equation*}
		\|f\|_{V^{s}(J;X)} \lesssim \|f\|_{R^{s}(J;X)}.
	\end{equation*}
\item  $V^{s}(J;X)\hookrightarrow R^{t}(J;X)$ and for all $f\in V^{s}(J;X)$ and $t>s$,
\begin{equation*}
		\|f\|_{R^{t}(J;X)} \lesssim \|f\|_{V^{s}(J;X)}.
	\end{equation*}
\item $C^{\frac1s}(J;X)\hookrightarrow V^s(J;X)$ and for all $f\in C^{\frac1s}(J;X)$ we have
	\begin{equation*}
		\|f\|_{V^s(J;X)}\leq \|f\|_{L^\infty(J;X)} + |J|^{\frac1s} [f]_{C^{1/s}(J;X)}.
	\end{equation*}
\end{enumerate}
\end{lemma}

The next lemma gives an interpolation inclusion for the $V^s$ spaces. A result in the converse direction for both $R^s$ and $V^s$ can be found in \cite{alv2}.
\begin{lemma}\label{lemma:interpolation}
Let $X_0,X_1$ be Banach spaces and $J \subseteq\R$. Let $\theta\in[0,1]$, $s_0,s_1\geq 1$ and set $\frac{1}{s_{\theta}}=\frac{1-\theta}{s_0}+\frac{\theta}{s_1}$. Then
$$[V^{s_0}(J;X_0),V^{s_1}(J;X_1)]_{\theta}\hookrightarrow V^{s_{\theta}}(J;[X_0,X_1]_{\theta}).$$
\end{lemma}

\begin{proof}
Let $X$ be a Banach space and $s\in[1,\infty]$. Let $P:= (t_0,  \cdots, t_{N})$ with $t_0,\ldots,t_N \in J$ such that $t_0<\cdots<t_N$.
Define the operator $\Phi_P:V^s(J;X)\to\ell^s(X)$ as
$$\Phi_Pf:=\left(f(t_0),f(t_{1})-f(t_0),\cdots,f(t_{N})-f(t_{N-1})\right).$$
Then $$\left\|\Phi_Pf\right\|_{\ell^s(X)}= \has{\left\|f(t_0)\right\|^s_{X} + \sum_{i=0}^{N-1}\left\|f(t_{i+1})-f(t_i)\right\|^s_{X}}^{\frac{1}{s}},$$
so taking the supremum over all such $P$, we have the following equivalence for the norm \eqref{Vnorm} of $V^s(J;X)$:
\begin{align}\label{3.6}
\left\|f\right\|_{V^s(J;X)}\eqsim_s \sup_P\left\|\Phi_Pf\right\|_{\ell^s(X)}.
\end{align}
Now note that $\Phi_P$ is a bounded linear operator from $V^{s_0}(J;X_0)$ to $\ell^{s_0}(X_0)$ and from $V^{s_1}(J;X_1)$ to $\ell^{s_1}(X_1)$. Therefore, by complex interpolation and \cite[Theorem 2.2.6]{HNVW1},
$\Phi_P$ is a bounded linear operator from $[V^{s_0}(J;X_0),V^{s_1}(J;X_1)]_{\theta}$ to $$[\ell^{s_0}(X_0),\ell^{s_1}(X_1)]_{\theta}=\ell^{s_{\theta}}([X_0,X_1]_{\theta}).$$
This combined with \eqref{3.6} yields
$$ \left\|f\right\|_{V^{s_{\theta}}(J;[X_0,X_1]_{\theta})}\lesssim\sup_P\left\|\Phi_Pf\right\|_{\ell^{s_{\theta}}([X_0,X_1]_{\theta})} \lesssim\left\|f\right\|_{[V^{s_0}(J;X_0),V^{s_1}(J;X_1)]_{\theta}},$$
finishing the proof.
\end{proof}

\subsection{Besov spaces and Triebel--Lizorkin spaces}\label{ss:BTL}
Below we recall the definitions of the vector-valued Besov and Triebel--Lizorkin spaces with $A_\infty$-weights.
 For further details on Besov and Triebel--Lizorkin spaces, the reader is referred to  \cite{Tri83, Tr1}.  For the unweighted case in the vector-valued setting, see \cite{AmannvolII,HNVW3,SchmSiunpublished,TriebelSpectra}. The weighted scalar-valued case was studied extensively in \cite{Bui82}, while the weighted vector-valued case can be found in \cite{MeyVer1}.

Fix $\varphi\in \Schw(\R)$ such that the Fourier transform $\wh{\varphi}$ of $\varphi$ satisfies
\[
 0\leq \wh{\varphi}(\xi)\leq 1, \quad  \xi\in \R, \qquad  \wh{\varphi}(\xi) = 1 \ \text{ if } \ |\xi|\leq 1, \qquad  \wh{\varphi}(\xi)=0 \ \text{ if } \ |\xi|\geq \frac32.
\]
Let $(\varphi_k)_{k\geq 0} \subseteq \Schw(\R)$ be defined by
\[
\wh{\varphi}_0 = \wh{\varphi}, \qquad \wh{\varphi}_1 = \wh{\varphi}(\cdot/2) - \wh{\varphi}, \qquad \wh{\varphi}_k = \wh{\varphi}_1(2^{-k+1} \cdot), \quad k\geq 2.
\]
Given $p \in (1,\infty)$, $q\in [1,\infty]$, $s\in \R$ and $w\in A_\infty$,  for $f\in {\mathscr S}'(\R;X)$ we set
\[ \|f\|_{B_{p,q}^s (\R,w;X)} = \Big\| \big( 2^{ks}\varphi_k* f\big)_{k\geq 0} \Big\|_{\ell^q(L^p(\R,w;X))},\]
\[ \|f\|_{F_{p,q}^s (\R,w;X)} = \Big\| \big( 2^{ks}\varphi_k* f\big)_{k\ge 0} \Big\|_{L^p(\R,w;\ell^q(X))}. \]
The spaces of tempered distributions for which the above norms are finite are denoted by $B_{p,q}^s (\R,w;X)$ and the $F_{p,q}^s (\R,w;X)$, respectively, which are Banach spaces. Any other $\wt{\varphi}$ satisfying the above leads to an equivalent norm.

\section{\texorpdfstring{$\mathcal{R}$}{R}-boundedness}\label{sec:m-R}
In this section, we briefly introduce the basic properties of $\mathcal{R}$-boundedness, which play a key role in our main result. For details we refer to \cite[Chapter 8]{HNVW2}.

\begin{definition}\label{5.9}
Let $X$ and $Y$ be Banach spaces and $(\varepsilon_n)_{n\ge 1}$ be a Rademacher
sequence. We say that a family of operators  $\mathcal{T}\subseteq\mathcal{L}(X,Y)$
 is \emph{$\mathcal{R}$-bounded} if there exists a constant $C> 0$ such that for all finite sequences $(T_n)_{n=1}^N$ in $\mathcal{T}$ and $(x_n)_{n=1}^N$ in $X$, we have
$$\E\Big\|\sum_{n=1}^N\varepsilon_nT_nx_n\Big\|_Y\leq C\, \E\Big\|\sum_{n=1}^N\varepsilon_nx_n\Big\|_X.$$
The least admissible constant in the inequality is called the $\mathcal{R}$-bound, denoted by $\mathcal{R}(\mathcal{T})$.
\end{definition}
In the above one can replace the $L^1$-moment with any $L^p$-moment for  $p\in [1, \infty)$ by the Kahane-Khintchine inequalities. By \cite[Propositions 8.1.21 and 8.1.22]{HNVW2}, if $\mathcal{T}$ is $\mathcal{R}$-bounded in $\mathcal{L}(X,Y)$, then the closure of the convex hull of $\mathcal{T}$, denoted as $\overline{\conv}(\mathcal{T})$, is also $\mathcal{R}$-bounded with
\begin{align}\label{5.27}
\mathcal{R}(\mathcal{T})=\mathcal{R}(\overline{\conv}(\mathcal{T})).
\end{align}

In the next result, we show that for any interval $J\subseteq\R$, functions in $R^{r}(J;\mathcal{L}(X,Y))$ and $V^{r}(J;\mathcal{L}(X,Y))$  have $\mathcal{R}$-bounded range under suitable conditions on $X$ and $Y$.
\begin{proposition}\label{R-bdd of R}
Let $X$, $Y$ be Banach spaces and $\mathcal{J}$ be a collection of  disjoint intervals in $\R$. Suppose that $X$ has cotype $q\in[2,\infty]$ and $Y$ has type $t\in[1,2]$. Let $\frac1r := \frac1t-\frac1q$. Then every $f\in \ell^{r}(R^{r}(\mathcal{J};\mathcal{L}(X,Y)))$ has $\mathcal{R}$-bounded range with
\[\mathcal{R}(\Ran(f)) \lesssim_{X,Y,q,t}  \|f\|_{\ell^{r}(R^{r}(\mathcal{J};\mathcal{L}(X,Y)))}.\]
Moreover, for $s \in [1,r)$ we have that every $f\in \ell^r(V^s(\mathcal{J};\mathcal{L}(X,Y)))$ has $\mathcal{R}$-bounded range with
\begin{align*}
	\mathcal{R}(\Ran(f))&\lesssim_{X,Y,q,s,t} \|f\|_{\ell^r(V^{s}(\mathcal{J};\mathcal{L}(X,Y)))},
\end{align*}
\end{proposition}
\begin{proof}
If $r=\infty$, then $t=q=2$, so every uniformly bounded family is $\mathcal{R}$-bounded (see \cite[Theorem 8.1.3]{HNVW2}). Now suppose that $r\in [1, \infty)$.  Let $J\in \mathcal{J}$ and $\mathcal{I}$ be a family of mutually disjoint subintervals of $J$. For $a\in R^r_{\mathrm{at}}(J;\mathcal{L}(X,Y))$ write $a=\sum_{I\in \mathcal{I}}c_I\one_I$, where
 $\{c_I\}_{I\in \mathcal{I}}\subseteq \mathcal{L}(X,Y)$. By \cite{vanGaans} (see also \cite[Proposition 8.1.20]{HNVW2}) we have
\begin{align*}
	\mathcal{R}(\Ran(a))=\mathcal{R}(\{c_{I}:{I\in \mathcal{I}}\})\lesssim \has{\sum_{I\in\mathcal{I}}\|c_I\|^r}^{1/r}\leq 1.
\end{align*}
Now take $f\in \ell^{r}(R^{r}(\mathcal{J};\mathcal{L}(X,Y)))$. Then $f|_J\in R^r(J;\mathcal{L}(X,Y))$ and thus
there exist sequences $(\lambda_k)_{k\ge 1}\subseteq\mathbb{C}$ and $(a_k)_{k\ge 1}\subseteq R^r_{\mathrm{at}}(J;\mathcal{L}(X,Y))$ such that $f|_J=\sum_{k=1}^\infty\lambda_ka_k$. From the triangle inequality one sees that
\begin{align*}
	\mathcal{R}(\Ran(f|_J))\leq \sum_{k=1}^\infty|\lambda_k |\cdot\mathcal{R}(\Ran(a_k))\lesssim \sum_{k= 1}^\infty|\lambda_k|.
\end{align*}
and hence, taking the infimum over $(\lambda_k)_{k\ge 1}$ and  $(a_k)_{k\ge 1}$ and employing \cite{vanGaans} once more on $\Ran(f)=\cup_{J\in\mathcal{J}}\Ran(f|_J)$, we obtain
\begin{align*}
	\mathcal{R}(\Ran(f))\lesssim \Big(\sum_{J\in\mathcal{J}}\mathcal{R}\hab{\Ran(f|_J)}^r\Big)^{1/r}\lesssim \|f\|_{ \ell^{r}(R^{r}(\mathcal{J};\mathcal{L}(X,Y)))},
\end{align*}
finishing the proof for $R^r$. The result for $V^s$ now follows by Lemma \ref{lemma:5.17}.
\end{proof}

Note that the result also holds if $\mathcal{J}$ consists of only one interval. Moreover, that interval could be $\R$, in which  case one obtains that functions in $V^{s}(\R;\mathcal{L}(X,Y))$ for $s \in [1,r)$ have $\mathcal{R}$-bounded range.

In \cite{HytVer} sufficient conditions for $\mathcal{R}$-boundedness are given in terms of Besov and H\"older regularity under the assumption that $X$ has cotype $q$ and $Y$ has type $t$. The space used there is for instance $B^{d/r}_{r,1}(\R^d;\calL(X,Y))$ and the latter embeds into the bounded continuous functions with values $\calL(X,Y)$.   Since the functions in $R^r$ and $V^{s}$ are not necessarily continuous, Proposition \ref{R-bdd of R} does not follow from \cite{HytVer}.

In the converse direction for $d=1$, this raises the natural question whether $B^{1/r}_{r,1}(\R;E)$ embeds into $R^r(\R;E)$, which we will show in Lemma \ref{lem:BesovV} below. As a consequence, the result of Proposition \ref{R-bdd of R} implies the result in \cite{HytVer} for $d=1$. Our proof can be  extended to $d\geq 2$, which we leave to the interested reader.

\begin{lemma}\label{lem:BesovV}
Let $E$ be a Banach space and let $r\in [1, \infty)$. Then for $r>1$ we have
\[B^{1/r}_{r,1}(\R;E)\hookrightarrow R^{r}(\R;E)\cap L^r(\R;E)\hookrightarrow V^{r}(\R;E)\cap L^r(\R;E)\hookrightarrow B^{1/r}_{r,\infty}(\R;E)\]
and for $r=1$ we have
\[B^{1}_{1,1}(\R;E)\hookrightarrow V^{1}(\R;E)\cap L^1(\R;E)\hookrightarrow B^{1}_{1,\infty}(\R;E)\]
\end{lemma}
Note that one does not have $B^{1}_{1,1}(\R)\hookrightarrow R^1(\R)$. Indeed, functions in $R^1(\R)$ can only take countably many values.  

\begin{proof} Since trivially $B^{1/r}_{r,1}(\R;E)\hookrightarrow L^r(\R;E)$, for the first embedding in the case $r>1$ we only need to show $B^{1/r}_{r,1}(\R;E) \hookrightarrow  R^{r}(\R;E)$. We will use the atomic decomposition of Besov spaces (see \cite[Section 3]{SSS12}). Indeed, any $f \in B^{1/r}_{r,1}(\R;E)$ can be written as
\begin{equation}\label{eq:BesovAtomic}
    f = \sum_{j=0}^\infty \sum_{k \in \Z} \mu_{j,k} a_{j,k},
\end{equation}
where the $\mu_{j,k}$ satisfy
\[
    \|f\|_{B^{1/r}_{r,1}(\R;E)} \eqsim_r \sum_{j=0}^\infty \Big( \sum_{k \in \Z} |\mu_{j,k}|^r \Big)^{1/r},
\]
and the atoms $a_{j,k}$ are supported on a fixed dilate of the interval around $k$ with length $2^{-j}$. Furthermore, they satisfy $\nrm{a_{j,k}}_{L^\infty(\R;E)} \leq 1$ and $\nrm{a_{j,k}'}_{L^\infty(\R;E)} \leq 2^j$.
Note that $$\nrm{f+g}_{R^r(\R;E)}^r \leq \nrm{f}_{R^r(\R;E)}^r+\nrm{g}_{R^r(\R;E)}^r$$ for disjointly supported $f,g\colon \R\to E$. Therefore, since for fixed $j\geq 0$ the supports of $a_{j,k}$ have bounded overlap, we can estimate
\begin{align*}
  \|f\|_{R^{r}(\R;E)} \lesssim_r \sum_{j=0}^\infty{\big( \sum_{k \in \Z} \|a_{j,k}\mu_{j,k}\|_{R^r(\R,E)}^r \big)}^{1/r}
\end{align*}
Hence, it remains to show that $\|a_{j,k}\|_{R^r(\R,E)} \lesssim 1$ independent of $j,k$. Denoting the support of $a_{j,k}$ by $I$, this is a direct consequence of (see Lemma \ref{lemma:5.17})
\[\|a_{j,k}\|_{R^r(\R,E)} \lesssim \|a_{j,k}\|_{V^1(\R,E)} \lesssim \|a_{j,k}\|_{L^\infty(I;E)} + |I| [a_{j,k}]_{C^1(I;E)}\lesssim 1.\]
This finishes the proof of $B^{1/r}_{r,1}(\R;E) \hookrightarrow  R^{r}(\R;E)$. The proof of $B^{1}_{1,1}(\R;E)\hookrightarrow V^{1}(\R;E)$ is identical.

The second embedding in the first claim follows from Lemma \ref{lemma:5.17}. To prove the final embedding of both claims, take $r \in [1,\infty)$. Using the well-known difference norm of the Besov space (see \cite[Section 2.5.12]{Tri83}) it suffices to show that for all $h>0$,
\begin{align*}
 \int_{\R} \|f(x+h) - f(x)\|^r \dd x\leq h [f]_{V^r(\R;E)}^r.
\end{align*}
To check this, note that
\begin{align*}
 \int_{\R} \|f(x+h) - f(x)\|^r dx & = \sum_{j\in \Z} \int_{jh}^{(j+1)h} \|f(x+h) - f(x)\|^r \dd x
 \\ & =  \sum_{j\in \Z} \int_{0}^{h} \|f(y+(j+1)h) - f(y+jh)\|^r \dd y
\\ & =  \int_{0}^{h} \sum_{j\in \Z}  \|f(y+(j+1)h) - f(y+jh)\|^r \dd y
\\ & \leq \int_{0}^{h} [f]_{V^r(\R;E)}^r\dd y = h [f]_{V^r(\R;E)}^r,
\end{align*}
finishing the proof of the second embedding.
\end{proof}
Similar results to Lemma \ref{lem:BesovV} for bounded intervals can in the scalar-valued case  be found in \cite{Ros09} with generalizations to the vector-valued setting in \cite{LPT20}. Furthermore, results for  $W^{s,r} = B^{s}_{r,r}$ with $s>1/r$ can be found in \cite[Corollary A.3]{FrizVictoir}.

\begin{remark}
From the fact that the Besov and Triebel--Lizorkin spaces contain unbounded functions (see \cite[Theorem 4.6.4/1]{RS96}) one sees that
\begin{align*}
B^{1/r}_{r,q}(\R)&\not\hookrightarrow V^{r}(\R) \ \ \text{for all $r\in [1, \infty)$ and $q\in (1,\infty]$},\\ 
F^{1/r}_{r,q}(\R)&\not\hookrightarrow V^{r}(\R) \ \ \text{for all $r\in (1,\infty)$ and $q \in[1,\infty]$}.
\end{align*}
This proves the sharpness of the $B^{1/r}_{r,1}$-embeddings in Lemma \ref{lem:BesovV}.

The embeddings into $B^{1/r}_{r,\infty}$ are also sharp. Indeed, one has
\begin{align*}
V^{r}(\R)\cap L^r(\R)\not\hookrightarrow B^{1/r}_{r,q}(\R) \ \ \text{for all $r\in [1, \infty]$ and $q\in [1, \infty)$,}
\\ 
V^{r}(\R)\cap L^r(\R)\not\hookrightarrow F^{1/r}_{r,q}(\R) \ \ \text{for all $r\in [1, \infty)$ and $q\in [1, \infty]$.}
\end{align*}
Indeed, one has that $f = \one_{[0,1]}\in V^{r}(\R)\cap L^r(\R)$, but not in $B^{1/r}_{r,q}(\R)$ for $q<\infty$ (see \cite[Lemma 4.6.3/2]{RS96}). To obtain the result for $F^{1/r}_{r,q}(\R)$, it suffices to consider $q=\infty$ and for this case the argument of \cite[Example 14.6.33]{HNVW3} can be used. \end{remark}

\section{Weighted, vector-valued variational Carleson estimates}\label{sec:m-carleson}
In this section, we prepare for our main result by establishing a weighted version of the vector-valued variational Carleson estimate.  We start by extending the boundedness of the vector-valued Carleson operator from \cite{HL} to the weighted setting.
For $f \in \mathcal{S}(\R;X)$ and $a \in \R$ define
$$
\mathcal{C}_a f(x):= \int_{-\infty}^a \widehat{f}(\xi)e^{2\pi i \xi x}\ud\xi , \quad x \in \R.
$$
The \emph{Carleson maximal operator}  is defined as
$$
\mathcal{C}_{*}f(x):= \sup_{a\in \R} \, \left\|\mathcal{C}_{a}f(x)\right\|_X, \quad  x\in \R.
$$

\begin{proposition}\label{9.4}
Let $\theta \in (0,1]$, let $X$ be a $\theta$-intermediate $\UMD$ Banach space and let $p\in(1,\infty)$. There exists a non-decreasing function $\phi_{X,p}:[1,\infty)\to [1,\infty)$  such that for all $w\in A_p$ and $f\in L^p(\R,w;X)$ we have
\begin{align*}
	\norm{\mathcal{C}_{*}f}_{L^p(\R,w)}\leq \phi_{X,p}([w]_{A_p}) \left\|f\right\|_{L^p(\R,w;X)}.
\end{align*}
\end{proposition}
\begin{proof}
To avoid strong measurability issues, let $(q_j)_{j\geq 1}$ be a dense sequence in $\R$, and set $Q_n = \{q_1, \ldots, q_n\}$ for $n\geq 1$. Since $\mathcal{C}_a f=f - S_{[a,\infty)}f$, by density, monotone convergence and continuity, it suffices to prove that there exists a non-decreasing function $\phi_{X,p}\colon[1,\infty)\to [1,\infty)$  such that for all $n\geq 1$, $w\in A_p$ and $f \in \mathcal{S}(\R;X)$,
\begin{align}\label{9.5}
	\nrms{\sup_{a\in Q_n} \|S_{[a,\infty)}f\|_X}_{L^p(\R,w)}\leq  \phi_{X,p}([w]_{A_p}) \left\|f\right\|_{L^p(\R,w;X)}.
\end{align}
Define $\mathcal{M}^af:=x \mapsto e^{-2\pi i x a}f(x)$  and note that
$$\left\|S_{[0,\infty)}\mathcal{M}^af\right\|_X=\left\|\mathcal{M}^aS_{[a,\infty)}f\right\|_X=\left\|S_{[a,\infty)}f\right\|_X.$$
Let $\mathcal{H}$ denote the Hilbert transform. Since $i\mathcal{H}+I=2S_{[0,\infty)}$, it is equivalent to show that there exists a non-decreasing function $\phi_{X,p}:[1,\infty)\to [1,\infty)$  such that for all $n\geq 1$, $w\in A_{p}$ and $f\in \mathcal{S}(\R;X)$,
\begin{align}\label{9.6}
	\nrms{\sup_{a\in Q_n} \left\|\mathcal{H}\mathcal{M}^a f\right\|_X}_{L^p(\R,w)}\leq \phi_{X,p}([w]_{A_p})\left\|f\right\|_{L^p(\R,w;X)}.
\end{align}

Fix $n\geq 1$. To prove \eqref{9.6}, we want to apply the sparse domination result in \cite{Lorist21}. To do so, let $p_0\in (1,\infty)$ and define $T:L^{p_0}(\R;X)\to L^{p_0}(\R;\ell^{\infty}(Q_n;X))$ by
$$Tf(x,a):=\mathcal{H}\mathcal{M}^a f(x),\quad x\in \R, \,a\in Q_n.$$
If $w\equiv 1$, \cite[Theorem 1.1]{HL} shows that \eqref{9.5} holds and therefore \eqref{9.6} also holds, which implies that $T$ is well-defined and bounded by density. For $f\in L^{p_0}(\R;X)$ we furthermore define
$$M^{\#}_T f(x):=\sup_{I\ni x}\esssup_{x_1,x_2\in I}\left\|T(f\one_{\R\backslash 5I})(x_1)-T(f\one_{\R\backslash 5I})(x_2)\right\|_{\ell^{\infty}(Q_n;X)},\quad x\in \R,$$
where the supremum is taken over all intervals $I$ containing $x$.

Fix $x\in \R$ and an interval $I$ containing $x$. Take $x_1, x_2\in I$ and denote the length of $I$ by $\varepsilon$. Note that for $y\in \R\backslash 5I$, $$|x-y|\eqsim|x_1-y|\eqsim|x_2-y| \gtrsim 2\varepsilon.$$ Then by \eqref{Hilbert-transform}, we have for $f \in \mathcal{S}(\R;X)$ that
\begin{align*}
	\nrmb{T(f\one_{\R\backslash 5I})(x_1)&-T(f\one_{\R\backslash 5I})(x_2)}_{\ell^{\infty}(Q_n;X)}\\
	&=\sup_{a\in Q_n}\frac{1}{\pi}\,\nrms{\int_{\R\backslash 5I}\has{\frac{1}{x_1-y}-\frac{1}{x_2-y}}e^{-2\pi i ya}f(y)\ud y}_X \\
	&\lesssim \int_{\R\backslash 5I}\frac{|x_1-x_2|}{|x_1-y||x_2-y|}\left\|f(y)\right\|_X\ud y   \\
	&\lesssim\varepsilon \int_{\R\backslash 5I}\frac{1}{|x-y|^2}\left\|f(y)\right\|_X\ud y\\
	&\lesssim\varepsilon \sum_{k=1}^{\infty}\int_{2^k \varepsilon<|x-y|\leq 2^{k+1}\varepsilon}\frac{1}{|x-y|^2}\left\|f(y)\right\|_X\ud y\\
	&\lesssim \sum_{k=1}^{\infty}\frac{1}{2^k}\frac{2}{2^{k+1}\varepsilon}\int_{|x-y|\leq 2^{k+1}\varepsilon}\left\|f(y)\right\|_X\ud y\\
	&\lesssim\sum_{k=1}^{\infty} \frac{1}{2^k}\cdot M\hab{\left\|f\right\|_X}(x)\lesssim M\hab{\left\|f\right\|_X}(x),
\end{align*}
where  $M$ is the Hardy-Littlewood maximal operator. Therefore we obtain
$$M^{\#}_Tf(x)\lesssim M\ha{\left\|f\right\|_X}(x),\quad x\in \R.$$
Since $M$ is weak $L^1$-bounded by \cite[Theorem 2.3.2]{HNVW1}, $M^{\#}_T$ is weak $L^1$-bounded as well.
We have now checked all assumptions of \cite[Theorem 3.2]{Lorist21} with $p_1=p_0$, $p_2=1$, $r=1$ and $(S,d,\mu)$ is $\R$ with the Euclidean metric and the Lebesgue measure. By \cite[Corollary 1.2]{Lorist21}, we conclude for all $p \in (p_0,\infty)$, $w\in A_{{p}/{p_0}}$ and $f\in {L^p(\R,w;X)}$,
\begin{align}\label{9.221}
	\nrms{\sup_{{a\in Q_n}} \nrmb{ \mathcal{H}\mathcal{M}^a f}_X}_{L^p(\R,w)}\lesssim  {\frac{p^2}{p-p_0}}\cdot C_{X,p_0} \cdot  [w]_{A_{{{p}/{p_0}}}}^{\max\{\frac{1}{p-p_0},1\}} \left\|f\right\|_{L^p(\R,w;X)},
\end{align}
where the factor ${\frac{p^2}{p-p_0}}$ arises from the proof of \cite[Proposition 4.1]{Lorist21} and
	$C_{X,p_0}$   the constant in \eqref{9.6} for the case $w\equiv1$, i.e. the constant from \cite[Theorem 1.1]{HL}. By inspection of the proof, it is clear that $C_{X,p_0}\to \infty$ as $p_0 \downarrow 1$.
	
Now, for fixed $w \in A_p$, let $p_0 \in (1,p)$ be such that $\frac{p}{p_0} = p-\varepsilon$ with $\varepsilon>0$ satisfying the conditions in Lemma \ref{self-improvement}. Combining this with \eqref{9.221}, we conclude that there is a non-decreasing function $\phi_{X,p}\colon[1,\infty)\to [1,\infty)$ such that \eqref{9.6} holds, finishing the proof.
\end{proof}

We will use Proposition \ref{9.4} to extend the boundedness of the vector-valued variational Carleson operator from \cite{amenta2022} to the weighted setting. Note that we only rely on Proposition \ref{9.4} (and thus the main result of \cite{HL}) in the proof. Hence, we also provide an alternative way to prove the unweighted boundedness of the vector-valued variational Carleson operator,  originally proven in \cite{amenta2022}.

For $q \in [1,\infty)$ and $f \in \mathcal{S}(\R;X)$ we define the \emph{variational Carleson operator} as
$$
\mathcal{C}_{*}^{q}f(x):=  \left[a \mapsto \mathcal{C}_{a}f(x)\right]_{V^{q}(\R;X)}=\underset{\mathcal{I}}{\sup}\Big(\underset{I\in \mathcal{I}}{\sum}\|S_{I}f(x)\|_X^{q}\Big)^{\frac{1}{q}}, \quad  x\in \R,
$$
where  the supremum is taken over all finite families of disjoint intervals $\mathcal{I}$ in $\R$. 

\begin{theorem}\label{thm:varcarleson}
Let $X$ be a  $\theta$-intermediate $\UMD$ Banach space for  $\theta\in(0,1]$. Let $q \in (2/{\theta},\infty)$ and $p \in (q',\infty)$. Then there exists a non-decreasing function $\phi_{X,p,q}:[1,\infty)\to [1,\infty)$  such that for all $w\in A_{{p}/{q'}}$ and $f\in L^p(\R,w;X)$,
\begin{align}\label{11.24}
	\left\| \mathcal{C}_{*}^{q}f\right\|_{L^p(\R,w)}\leq \phi_{X,p,q}([w]_{A_{{p}/{q'}}}) \left\|f\right\|_{L^p(\R,w;X)}.
\end{align}
\end{theorem}
In case $X=\C$ (and thus $\theta=1$) such results can be found in \cite{DiDoUr, DoLac, Lac07, ObSeTaThWr}.
\begin{proof}
By density it suffices to prove the result for $f\in \Schw(\R;X)$. Let $(Q_n)_{n\geq 1}$ be as in the proof of Proposition \ref{9.4}.
Let $n\in \mathbb{N}$, define $J_n:=[-n,n]\cap Q_n$ and
$$
\mathcal{C}_{*,n}^{q}f(x):=  \left[a \mapsto \mathcal{C}_{a}f(x)\right]_{V^{q}(J_n;X)}, \quad  x\in\R.
$$
By the monotone convergence theorem and continuity, it suffices to prove the statement for  $\mathcal{C}_{*,n}^{q}f$.

Fix $u\in (2, q)$. We start by proving an estimate on a Hilbert space $H$.  Note that  $H$ is isomorphic to $L^2$, so the $u'$-concavification $H^{u'}$ is a $\UMD$ Banach function space since $L^{2/{u'}}$ is a $\UMD$ Banach function space. For any $p_0>u'$, by \cite[Theorem 5.2]{ALV} and Proposition \ref{9.4}, there exists a non-decreasing function $\phi_{H,p_0,u}:[1,\infty)\to [1,\infty)$ such that for all $v\in A_{{p_0}/{u'}}$ and $f\in L^{p_0}(\mathbb{R},v;H)$,
\begin{align*}
	\left\|\mathcal{C}_{*,n}^uf\right\|_{L^{p_0}(\mathbb{R},v)}
	\leq\phi_{H,p_0,u}([v]_{A_{{p_0}/{u'}}})\left\| f\right\|_{L^{p_0}(\mathbb{R},v;H)},
\end{align*}
and	
\begin{align}\label{eq:Hilbertendpoint}
	\begin{array}{lll}
		\left\|(x,a) \mapsto \mathcal{C}_{a}f(x)\right\|_{L^{p_0}(\R,v;V^u(J_n;H))}
		&\leq \left\|\left\|a \mapsto \mathcal{C}_{a}f(\cdot)\right\|_{\ell^\infty(J_n;H)}\right\|_{L^{p_0}(\R,v)}\\
		&\qquad+\left\| \left[a \mapsto \mathcal{C}_{a}f(\cdot)\right]_{V^{u}(J_n;H)}\right\|_{L^{p_0}(\mathbb{R},v)}\\
		& \leq\phi_{H,p_0,u}([v]_{A_{{p_0}/{u'}}})\left\|f\right\|_{L^{p_0}(\mathbb{R},v;H)}.
	\end{array}
\end{align}

Let  $X_0$ be a $\UMD$ Banach space and let $H$ be a Hilbert space such that $X:=[X_0,H]_{\theta}$. Set $\theta_0 := \tfrac{u}{q}$.
Then by reiteration, there exists an $\eta\in(0,1)$ small enough such that $X_1:=[X_0, H]_{\eta}$ and
$$X=[X_0,H]_{\theta}=[X_1,H]_{\theta_0}.$$
Proposition \ref{9.4} yields for all $p_1 \in (1,\infty)$, there exists a non-decreasing function $\phi_{X_1,p_1}:[1,\infty)\to [1,\infty)$  such that for $v\in A_{p_1}$, $f\in L^{p_1}(\R,v;X_1)$,
\begin{align}\label{9.91}
	\begin{array}{lll}
		\left\|(x,a) \mapsto \mathcal{C}_{a}f(x)\right\|_{L^{p_1}(\mathbb{R},v;\ell^{\infty}(J_n;X_1))}
		\leq\phi_{X_1,p_1}([v]_{A_{p_1}})\left\|f\right\|_{L^{p_1}(\R,v;X_1)}.
	\end{array}
\end{align}
Fix $w \in A_{{p/q'}}$, choose $p_0 := \tfrac{p}{q'}u'>u'$ and
$p_1 := \tfrac{p}{q'}>1.$
Then
$$
\tfrac{1-\theta_0}{p_1}+\tfrac{\theta_0}{p_0} = \tfrac{1}{p},\quad \text{and}\quad \tfrac{\theta_0}{u} =\tfrac{1}{q},
$$
so by Lemma \ref{lemma:interpolation}  and \cite[Theorem 2.2.6]{HNVW1} we have
\begin{align*}		&[L^{p_1}(\mathbb{R},w;\ell^\infty(J_n;X_1)),L^{p_0}(\mathbb{R},w;V^u(J_n;H))]_{\theta_0}
	\hookrightarrow L^{p}(\mathbb{R},w;V^{q}(J_n;X)).
\end{align*}
Hence, using complex interpolation between \eqref{eq:Hilbertendpoint} and \eqref{9.91}, we conclude that there is a non-decreasing function $\phi_{X,p,q}:[1,\infty)\to [1,\infty)$ such that for all $f\in {L^p(\R,w;X)}$,
$$	\left\|{(x,a) \mapsto \mathcal{C}_{a}f(x)}\right\|_{L^p(\R,w;V^{q}(J_n;X))}\leq \phi_{X,p,q}([w]_{A_{p/q'}}) \left\|f\right\|_{L^p(\R,w;X)}$$
finishing the proof.
\end{proof}

As an immediate consequence of Theorem \ref{thm:varcarleson}, we obtain Theorem \ref{thm:RubioX}. That is, we also can conclude that for any family of disjoint intervals $\mathcal{I}$  in $\R$ and $f \in L^p(\R,w;X)$,
\begin{align}\label{eq:RubioX2}
  \Big\|\Big(\sum_{I\in \mathcal{I}}\|S_{I}f\|^{q}_X\Big)^{1/q}\Big\|_{L^p(\R,w)}\leq \phi_{X,p,q}([w]_{A_{p/q'}}) \|f\|_{L^p(\R,w;X)}.
\end{align}

Conversely, one could ask what geometric properties of $X$ are implied by \eqref{eq:RubioX2}. In the next proposition, we show that it implies cotype $q$.
\begin{proposition}\label{prop:RubioBGT}
Let $X$ be a Banach space and $w$ be a weight. Let  $p\in[1,\infty)$ and $q\in[2,\infty)$. Suppose that
there is a constant $C>0$ such that for all  $f\in L^p(\R,w;X)$ and all families of disjoint intervals $\mathcal{I}$ in $\R$
\begin{align*}
	\Big\|\big(\sum_{I\in \mathcal{I}}\|S_{I}f\|^{q}_X\big)^{\frac{1}{q}}\Big\|_{L^p(\R,w)}\leq C \|f\|_{L^p(\R,w;X)}.
\end{align*}
Then $X$ has cotype $q$.
\end{proposition}
\begin{proof}
Take a finite sequence  $(x_n)_{n=1}^N$  in $X$. Let $(e_n)_{n\in \Z}$ be the trigonometric system on $[0,1]$ and extend it periodically to $\R$. Let $(\varepsilon_n)_{n\geq 1}$ be a sequence of complex Rademacher (or Steinhaus) random variables on $\Omega$ (see \cite{HNVW2}). Define  $f := \varphi \cdot \sum_{n=1}^N e_{3n} \varepsilon_n x_n$, where $\varphi\in \Schw(\R)$ is nonzero and  $\supp(\wh{\varphi})\subseteq [0,1]$. Let $I_{n} = [3n-1, 3n+1]$. Then $S_{I_n} f =\varphi \cdot e_{3n} \varepsilon_n x_n$  pointwise in $\Omega$ and thus
\begin{align*}
\Big\|\Big(\sum_{n=1}^N\|S_{I_n}f\|^{q}_X\Big)^{1/q}\Big\|_{L^p(\R,w)} =  \|\varphi\|_{L^p(\R,w)} \Big(\sum_{n=1}^N \|x_n\|^{q}_X\Big)^{1/q}.
\end{align*}
Therefore, applying \eqref{eq:RubioX2} and taking $p$-th moments on both sides we find that
\begin{align*}
\|\varphi\|_{L^p(\R,w)} \Big(\sum_{n=1}^N \|x_n\|^{q}_X\Big)^{1/q} &\lesssim \hab{\E\,\|f\|_{L^p(\R,w;X)}^p}^{1/p}
\\ & \lesssim \has{\E\,\Big\|\varphi \cdot \sum_{n=1}^N e_{3n} \varepsilon_n x_n\Big\|_{L^p(\R,w;X)}^p}^{1/p}
\\ &\lesssim \has{\E\,\Big\|\varphi \cdot \sum_{n=1}^N \varepsilon_n x_n\Big\|_{L^p(\R,w;X)}^p}^{1/p}
\\ & \lesssim \|\varphi\|_{L^p(\R,w)} \E\,\Big\| \sum_{n=1}^N \varepsilon_n x_n\Big\|_{X},
\end{align*}
where we used that for fixed {$t\in \R$, $(\varepsilon_n e_{3n}(t))_{n=1}^N$ } has the same distribution as $(\varepsilon_n)_{n=1}^N$ due to the unimodular invariance. Therefore, $X$ has cotype $q$.
\end{proof}

The $\LPR_p$ property introduced in \cite{BGT03} asserts that for any $p\in[2,\infty)$ and any family of disjoint intervals $\mathcal{I}$, a Rademacher sequence $(\varepsilon_I)_{I \in \mathcal{I}}$ and $f \in L^p(\R;X)$ we have
\begin{align}\label{eq:LPR}
 \E\, \Big\|\sum_{I\in \mathcal{I}}\varepsilon_I S_{I}f\Big\|_{L^p(\R;X)}\leq C_{X,p} \|f\|_{L^p(\R;X)}.
\end{align}
By \cite{HTY} we know that \eqref{eq:LPR} implies type $2$. In contrast, our vector-valued Littlewood--Paley--Rubio de Francia type estimate \eqref{eq:RubioX2} does not imply type $2$, since it holds for any $\theta$-intermediate UMD space. However, it does imply cotype $q$.
If $X$ has the $\LPR_p$ property and  cotype $q$, then the estimate \eqref{eq:RubioX2} holds for $w=1$. Indeed,  by cotype $q$, the Kahane--Khintchine inequalities and $\LPR_p$ we can estimate
\[\Big\|\big(\sum_{I\in \mathcal{I}}\|S_{I}f\|^{q}_X\big)^{\frac{1}{q}}\Big\|_{L^p(\R)} \lesssim \E\Big\|\sum_{I\in \mathcal{I}}\varepsilon_I S_{I}f\Big\|_{L^p(\R;X)}  \lesssim \|f\|_{L^p(\R;X)}.\]
Therefore, in the case that \eqref{eq:RubioX2} holds, the $\LPR_p$-property is quantitatively stronger than \eqref{eq:RubioX2}. On the other hand, beyond spaces which are isomorphic to a closed subspace of a Banach function space, there are no spaces known with the $\LPR_p$-property. In contrast \eqref{eq:RubioX2} holds for many Banach spaces, as shown in Theorem \ref{thm:varcarleson}.

For a Banach function space $X$, a different extension of Rubio de Francia's original estimate \eqref{eq:Rubioscalar} to the vector-valued setting was studied in \cite{ALV, Krol}. Indeed, in a Banach function space, one can study the estimate
\begin{align*}
  \Big\|\Big(\sum_{I\in \mathcal{I}}|S_{I}f|^{q}\Big)^{\frac{1}{q}}\Big\|_{L^p(\R,w;X)}\lesssim \|f\|_{L^p(\R,w;X)}.
\end{align*}
for any family of disjoint intervals $\mathcal{I}$  in $\R$. In a Banach function space, the result of \cite[Theorem 1.2]{ALV} is stronger than  Theorem \ref{thm:RubioX}.
Indeed, if $X  = [X_0, L^2(S)]_{\theta}$ with $X_0$ a UMD Banach function space over a measure space $(S,\mu)$, then $X$ is $q$-concave and $q'$-convex for  $q =  2/\theta$. Furthermore, the $q'$-concavification of $X$ is given by
$$
X^{q'} = [X_0,L^1(S)]_{\frac{\theta}{2-\theta}}.
$$
By self-improvement of the $\UMD$ property of $X_0$ (see \cite[Theorem 4]{Rubio86}), one can replace $L^1(S)$ by $L^{1+\varepsilon}(S)$ for some $\varepsilon>0$ and   therefore $X^{q'}$ has  $\UMD$. Hence, the assumptions of Theorem \ref{thm:RubioX} imply the assumptions of \cite[Theorem 1.2]{ALV}. Moreover, by $q$-concavity we note that
\begin{align*}
\Big\|\big(\sum_{I\in \mathcal{I}}\|S_{I}f\|^{q}_X\big)^{\frac{1}{q}}\Big\|_{L^p(\R)} \lesssim \Big\|\big(\sum_{I\in \mathcal{I}}|S_{I}f|^{q}\big)^{\frac{1}{q}}\Big\|_{L^p(\R;X)},
\end{align*}
so the conclusion of \cite[Theorem 1.2]{ALV} implies the conclusion of  Theorem \ref{thm:RubioX}. We refer to \cite[Section 5]{amenta2022} for a  comparison of Theorem \ref{thm:varcarleson} and the Banach function space case in  \cite[Theorem 5.2]{ALV}.

\section{Main Result}\label{sec:m-main}

\subsection{Fourier multiplier on \texorpdfstring{$L^p$}{Lp}-spaces}
Having shown the weighted, vector-valued variational Carleson estimate, and as a consequence the required Rubio de Francia type inequality, we now turn to our first main result on Fourier multipliers.

\begin{theorem}\label{thm:boundedsupp}
Let $\theta\in(0,1]$. Let $X$ be a  $\theta$-intermediate $\UMD$ Banach space and let $Y$ be a Banach space. Let  $s\in[1,\frac{2}{2-\theta})$, $p\in (s,\infty)$ and let $m\in V^{s}(\R;\mathcal{L}(X,Y))$. There exists a non-decreasing function $\phi_{X,Y,p,s}\colon[1,\infty)\to[1,\infty)$ such that for all $w\in A_{p/s}$,
$$\|T_m\|_{\mathcal{L}(L^p(\R,w;X),L^p(\R,w;Y))}\leq \phi_{X,p,s}([w]_{A_{p/s}})\cdot\|m\|_{V^{s}(\R;\mathcal{L}(X,Y))}.$$
\end{theorem}

Note that the range of $m$ as in Theorem \ref{thm:boundedsupp} is $\mathcal{R}$-bounded by Proposition \ref{R-bdd of R}. However, we will not use this fact in the proof.

\begin{proof}
We first prove the result for $m \in R^s(\R;\mathcal{L}(X,Y))$.  Write
$$m=\sum_{k=1}^{\infty}  \lambda_{k} a_{k},\quad a_{k}=\sum_{I\in \mathcal{I}_{k}}c_I \one_I,$$
where for $c_I\in \mathcal{L}(X,Y)$, $\mathcal{I}_{k}$ a family of mutually disjoint subintervals $I\subseteq \R$ such that $\sum_{I\in \mathcal{I}_k}\|c_I\|_{\mathcal{L}(X,Y)}^s\leq 1$ and
$$
\sum^{\infty}_{k=1}| \lambda_{k}| \leq \|m\|_{R^{s}(\R;\mathcal{L}(X,Y))}.
$$
Then we have for  $f \in L^p(\R,w;X)$ that
\begin{align*}
\|T_mf\|_{L^p(\mathbb{R},w;Y)}&=\Big\|\sum_{k=1}^{\infty} \lambda_{k}\sum_{I\in \mathcal{I}_{k}}c_IS_I f\Big\|_{L^p(\mathbb{R},w;Y)}
\\& \leq \sum_{k=1}^{\infty} |\lambda_{k}| \Big\|\sum_{I\in \mathcal{I}_{k}}c_IS_I f\Big\|_{L^p(\mathbb{R},w;Y)}
\\ & \lesssim \|m\|_{R^{s}(\R;\mathcal{L}(X,Y))} \cdot \sup_{k\in \N}\Big\|\sum_{I\in \mathcal{I}_{k}} c_IS_I f\Big\|_{L^p(\mathbb{R},w;Y)}
\end{align*}
For the latter we can estimate for fixed $k \in \N$ by H{\"o}lder's inequality
\begin{align*}
\Big\|\sum_{I\in \mathcal{I}_{k}}c_IS_I f\Big\|_Y & \leq \sum_{I\in \mathcal{I}_{k}}\|c_I\|_{\mathcal{L}(X,Y)}\|S_If\|_X
\leq \Big(\sum_{I\in \mathcal{I}_{k}}\|S_If\|_X^{s'} \Big)^{1/s'}.
\end{align*}
Therefore, applying Theorem \ref{thm:RubioX} we can conclude that
\begin{align*}
\|T_mf\|_{L^p(\mathbb{R},w;Y)} &\lesssim \|m\|_{R^{s}(\R;\mathcal{L}(X,Y))} \cdot\sup_{k\in\N} \Big\|\Big(\sum_{I\in \mathcal{I}_k}\|S_If\|_X^{s'} \Big)^{1/s'}\Big\|_{L^p(\mathbb{R},w)} \\ & \leq \|m\|_{R^{s}(\R;\mathcal{L}(X,Y))} \cdot \phi_{X,p,s}([w]_{A_{{p/s}} }) \|f\|_{L^p(\mathbb{R},w;X)},
\end{align*}
concluding the proof for $R^s$.

To obtain the result for $V^s$, note that by Lemma \ref{self-improvement} we can choose $s_1\in (s,\frac{2}{2-\theta}\wedge p)$ depending on $p,s$ and $[w]_{ A_{{p/s}}}$ such that  $w\in A_{{p/s_1}}$ and $[w]_{A_{{p/s_1}}}\lesssim_{p,s}[w]_{A_{{p/s}}}.$ Since $V^s\hookrightarrow R^{s_1}$ by Lemma \ref{lemma:5.17}, the result follows from the previous case applied with $s_1$ instead of $s$.
\end{proof}

Combining Theorem \ref{thm:boundedsupp} with the Littlewood--Paley theorem we obtain the following result for multipliers with some decay at zero and infinity (measured using an $\ell^r$-norm over $\Delta$). This is a quantified version of Theorem \ref{thm:mainintro}. Note that the second statement in Theorem \ref{thm:mainintro} follows by duality.
\begin{theorem}\label{thm:mainthm}
Let $\theta\in (0,1]$,  $t\in (1, 2]$, $q\in [2, \infty)$ and set $\frac1r := \frac1t-\frac1q$. Let $X$ be a {$\theta$-intermediate} $\UMD$ Banach space with cotype $q$ and $Y$ be a $\UMD$ Banach space with type $t$.
Let $s\in[1,\frac{2}{2-\theta})$ and suppose that  $m\in \ell^{r}(\dot{V}^{s}(\Delta;\mathcal{L}(X,Y)))$ has $\mathcal{R}$-bounded range.
Then for all $p\in (s,\infty)$,  there exists a non-decreasing function $\phi_{X,Y,p,q,s,t}\colon[1,\infty)\to[1,\infty)$ such that for all $w\in A_{p/s}$,
\begin{align*}
\|T_m\|_{\mathcal{L}(L^p(\R,w;X),L^p(\R,w;Y))} & \leq \phi_{X,Y,p,q,s,t}([w]_{A_{p/s}})\big(\|m\|_{ \ell^{r}(\dot{V}^{s}(\Delta;\mathcal{L}(X,Y)))} + \mathcal{R}(\Ran(m))\big).
\end{align*}
\end{theorem}

Note that the assumptions in Theorem \ref{thm:mainthm} imply that $r>s$. Indeed, by complex  interpolation we know that $X$ has cotype $q=\frac{2}{\theta}$ (see \cite[Proposition 7.1.3]{HNVW2}) and therefore $\frac{1}{r}\leq1-\frac{\theta}{2}<\frac{1}{s}$.

\begin{proof}
By rescaled Rubio de Francia extrapolation (see \cite[Corollary 3.14]{Cruz}) it suffices to consider $p=2$. We will first prove the result for $V^s$ instead of $\dot{V}^s$.

Take $m\in \ell^r(V^s(\Delta;\mathcal{L}(X,Y)))$, for $J \in \Delta$ define $k_J:= \|\one_J m\|_{V^{s}(J;\mathcal{L}(X,Y))}$ and set
$$C_m:= \|(k_J)_{J}\|_{\ell^r} =\|m\|_{ \ell^{r}(V^{s}(\Delta;\mathcal{L}(X,Y)))}.$$
Take $f\in L^2(\R,w;X)$  such that $\widehat{f}$ has compact support in $\R \setminus \cbrace{0}$, so that below the sum over $\Delta$ is finite.

By Proposition \ref{prop:LPw} and type $t$ of $L^2(\R,w;Y)$ we can write
\begin{align*}
\|T_m f\|_{L^2(\R,w;Y)}& \leq \phi_{Y}([w]_{A_2}) \mathbb{E}\,\Big\|\sum_{J\in \Delta} \varepsilon_J T_{\one_J m} S_J f \Big\|_{L^2(\mathbb{R},w;Y)}
\\ & \leq \phi_{Y,t}([w]_{A_2}) \Big(\sum_{J\in \Delta} \|T_{\one_Jm} S_J f \|_{L^2(\mathbb{R},w;Y)}^t\Big)^{1/t}
\end{align*}
Using Theorem \ref{thm:boundedsupp} we deduce
\[\|T_{\one_J m} S_J f \|_{L^2(\mathbb{R},w;Y)}\leq \phi_{X,Y,s}([w]_{A_{2/s}}) \cdot  k_J\cdot  \|S_J f\|_{L^2(\R,w;X)}.\]
Therefore, using H\"older's inequality, $A_{2/s} \subseteq A_{2}$ and cotype $q$ of $L^2(\R,w;X)$, we obtain
\begin{align*}
\|T_m f\|_{L^2(\R,w;Y)}&\leq \phi_{X,Y,s,t}([w]_{A_{2/s}}) \Big(\sum_{J\in \Delta} k_J^t\cdot  \|S_J f\|_{L^2(\R,w;X)}^t\Big)^{1/t}
\\ & \leq \phi_{X,Y,s,t}([w]_{A_{2/s}}) \cdot \|(k_J)_{J}\|_{\ell^r} \cdot \Big(\sum_{J\in \Delta}\|S_J f\|_{L^2(\R,w;X)}^q\Big)^{1/q}
\\ & \leq \phi_{X,Y,q,s,t}([w]_{A_{2/s}})\cdot  C_m \cdot \E\, \Big\|\sum_{J\in \Delta} \varepsilon_J S_J f\Big\|_{L^2(\R,w;X)}
\\ & \leq \phi_{X,Y,q,s,t}([w]_{A_{2/s}})\cdot C_m \cdot \|f\|_{L^2(\R,w;X)}.
\end{align*}
By density (see \cite[Lemma 3.3]{FHL}), this finishes the proof for $V^s$.

Now take $m\in \ell^r(\dot{V}^s(\Delta;\mathcal{L}(X,Y)))$ and define
\[\wt{m} = \sum_{J\in \Delta} \one_{J} (m - m_J),\]
where $m_J = |J|^{-1}\int_{J} m(y) \dd y$ for $J \in \Delta$. Then $[\wt{m}]_{V^s(J;\calL(X,Y))} = [m]_{V^s(J;\calL(X,Y))}$ for all $J \in \Delta$ and
\begin{align*}
\|\wt{m}\|_{L^\infty(J;\calL(X,Y))} &\leq \frac{1}{|J|}\int_J \|x\mapsto m(x) - m(y)\|_{L^\infty(J;\calL(X,Y))} \dd y
\\ &\leq [m]_{V^s(J;\calL(X,Y))}.
\end{align*}
Therefore, $\wt{m}\in\ell^r(V^s(\Delta;\mathcal{L}(X,Y)))$
	and 
\[\|T_{\wt{m}} f\|_{L^2(\R,w;Y)}\leq \phi_{X,Y,q,s,t}([w]_{A_{2/s}}) \cdot \|m\|_{ \ell^{r}(\dot{V}^{s}(\Delta;\mathcal{L}(X,Y)))}\cdot \|f\|_{L^2(\R,w;X)}.\]
To complete the proof it suffices to show boundedness of  $T_{\bar{m}}$ with $\bar{m} := \sum_{J\in \Delta} \one_{J} m_J$. Since $\Ran(m)$ is $\mathcal{R}$-bounded, by \eqref{5.27} we know that
$\overline{\conv}(\Ran(m))$ is also $\mathcal{R}$-bounded. Since $m_J\in \overline{\conv}(\Ran(m))$ (\cite[Proposition 1.2.12]{HNVW1}), 
by Proposition \ref{prop:LPw}  we have that
\begin{align*}
	\|T_{\bar{m}}f\|_{L^2(\R,w;Y)}&\leq  \phi_{Y}([w]_{A_{2}})\cdot \mathbb{E}\,\nrms{\sum_{J\in\Delta}\varepsilon_J m_{J} S_Jf}_{L^2(\R,w;Y)}\\  &\le \phi_{Y}([w]_{A_{2}})\cdot \mathcal{R}(\Ran(m)) \cdot \mathbb{E}\,\nrms{\sum_{J\in\Delta}\varepsilon_J S_Jf}_{L^2(\R,w;X)}\\
	&\leq \phi_{X,Y}([w]_{A_{2}})\cdot \mathcal{R}(\Ran(m))\cdot\|f\|_{L^2(\R,w;X)},
\end{align*}
finishing the proof.
\end{proof}

\subsection{Fourier multipliers on Besov and Triebel--Lizorkin spaces}
If $L^p$ is replaced by a Besov space or Triebel--Lizorkin space, our multiplier theorem simplifies. Indeed, type, cotype and $\mathcal{R}$-boundedness do not play any role anymore.
Moreover, no decay of $m$ is required. We state the result for the inhomogeneous Besov space $B^{\alpha}_{p,q}$. The reader is referred to the text below the theorem for the necessary changes in the homogeneous setting. We define
\begin{align*}
J_0&:= (-1,1),\\
  J_n&:= \cbrace{ \xi \in \R:2^{n-1}\leq |\xi|<2^{n}}, \quad n \in \N.
\end{align*}

\begin{theorem}[Fourier multipliers on Besov spaces]\label{thm:Besov}
Fix $\theta\in (0,1]$. Let $X$ be a $\theta$-intermediate $\UMD$ Banach space and $Y$ be a $\UMD$ Banach space.   Let $s\in[1,\frac{2}{2-\theta})$ and assume that $m \colon \R \to \mathcal{L}(X,Y)$ satisfies
\[C_m := \sup_{n \geq 0}\|m\|_{V^{s}(J_n;\mathcal{L}(X,Y))}<\infty.\]
Let $p\in (s, \infty)$ and $q\in [1, \infty]$ and $\alpha\in \R$.
Then there exists a non-decreasing function $\phi_{X,Y,p,q,s}\colon[1,\infty)\to[1,\infty)$ such that for all $w\in A_{p/s}$,
\begin{align*}
\|T_m\|_{\mathcal{L}(B^{\alpha}_{p,q}(\R,w;X),B^{\alpha}_{p,q}(\R,w;Y))}&\leq \phi_{X,Y,p,q,s}([w]_{A_{p/s}}) \cdot C_m.
\end{align*}
\end{theorem}
\begin{proof}
 Let $f\in B^{\alpha}_{p,q}(\R,w;X)$. Write $f_n = \varphi_n*f$ and $m_n = \wh{\varphi}_n m$, where $(\varphi_n)_{n\geq 0}$ is an inhomogeneous Littlewood--Paley sequence as in \cite[Chapter 14]{HNVW3}. Then by the properties of the $(\varphi_n)_{n\geq 0}$ we can write 
\begin{align*}
\|T_m f\|_{B^{\alpha}_{p,q}(\R,w;Y))} \leq  \sum_{j=-1}^1 \Big\|\Big(2^{\alpha k} \|T_{m_{n+j}} f_n\|_{L^p(\R,w;Y)}\Big)_{n\geq0}\Big\|_{\ell^q}, 
\end{align*}
where we set $m_{-1}=0$. 
Therefore, it suffices to show that for any $n \geq 0$,
\begin{align*}
  \|T_{m_n}g\|_{L^p(\R,w;Y)}\leq\phi_{X,Y,p,s}([w]_{A_{p/s}}) \cdot C_m \|g\|_{L^p(\R,w;X)}, \ \ \ g\in L^p(\R,w;X).
\end{align*}
By Theorem \ref{thm:boundedsupp} and the support conditions of $\wh{\varphi}_n$, it suffices to estimate $\|m_n\|_{V^{s}(\wt{J}_{n};\calL(X,Y))}$, where $\wt{J}_{n} :=  J_{n} \cup J_{n+1}$, and we note that $\supp(\wh{\varphi}_n)\subseteq \wt{J}_n$.
Since $|\wh{\varphi}_n|\leq 1$, it is clear that $$\|m_n\|_{L^\infty(\wt{J}_n;\mathcal{L}(X,Y))}\leq \|m\|_{L^\infty(\R;\mathcal{L}(X,Y))} \leq C_m.$$
Furthermore, we have
\begin{align*}
[m_n]_{V^{s}(\wt{J}_{n};\calL(X,Y))}\leq \|\wh{\varphi}_n\|_{L^\infty(\wt{J}_n)} \|m\|_{V^{s}(\wt{J}_{n};\calL(X,Y))} + \|\wh{\varphi}_n\|_{V^s(\wt{J}_n)} \|m\|_{L^\infty(\wt{J}_n;\calL(X,Y))} \lesssim C_m,
\end{align*}
finishing the proof.
\end{proof}

A similar result holds for the homogeneous Besov space if one uses $\Delta$ instead of $(J_n)_{n\geq 0}$. A result related to Theorem \ref{thm:Besov} under a Fourier type $s$ condition and using $m\in B^{1/s}_{s,1}(J;\calL(X,Y))$ was obtained in \cite{GW03Besov}. Note that the space $V^{s}(J;\calL(X,Y))$ is larger, as shown in Lemma \ref{lem:BesovV}. It, in particular, contains non-continuous multipliers.

Next, we move to the case of Triebel--Lizorkin spaces. Multiplier theorems in that setting can be found in \cite{Tri83} for the scalar case, and in \cite{BuKim} for the vector-valued setting. We obtain our result from the Besov space case using Rubio de Francia extrapolation.

\begin{theorem}[Fourier multipliers on Triebel--Lizorkin spaces]\label{thm:TL}
Let $\theta\in (0,1]$. Let $X$ be a $\theta$-intermediate $\UMD$ Banach space and $Y$ be a $\UMD$ Banach space.  Let $s\in[1,\frac{2}{2-\theta})$ and assume that $m \colon \R \to \mathcal{L}(X,Y)$ satisfies
\[C_m := \sup_{n \geq 0}\|m\|_{V^{s}(J_n;\mathcal{L}(X,Y))}<\infty.\]
Let $p,q\in (s,\infty)$ and $\alpha\in \R$. Then there exists a non-decreasing function $\phi_{X,Y,p,q,s}\colon[1,\infty)\to[1,\infty)$ such that for all $w\in A_{p/s}$,
\begin{align*}
&\|T_m\|_{\mathcal{L}(F^{\alpha}_{p,q}(\R,w;X),F^{\alpha}_{p,q}(\R,w;Y))} \leq\phi_{X,Y,p,q,s}([w]_{A_{p/s}})\cdot C_m.
\end{align*}
\end{theorem}
\begin{proof}
In the notation of the proof of Theorem \ref{thm:Besov}, we have to show that
\begin{align*}
  \|(\varphi_n *T_{m} f)_{n\geq 0}\|_{L^p(\R,w;\ell^q(Y))}\leq \phi_{X,Y,p,q,s}([w]_{A_{p/s}})\cdot C_m \|(\varphi_n * f)_{n\geq 0}\|_{L^p(\R,w;\ell^q(X))}.
\end{align*}
Since $q\in (s, \infty)$ by rescaled Rubio de Francia extrapolation (see \cite[Corollary 3.14]{Cruz}), it suffices to consider $p=q$. This case follows directly from Theorem \ref{thm:Besov}.
\end{proof}

\section*{Open problems}

\begin{problem}
Can one characterize the vector-valued Rubio de Francia type estimate in \eqref{eq:RubioX} or its converse in terms of other geometric conditions on $X$?
\end{problem}
Note that \eqref{eq:RubioX} implies that $S_I$ is bounded on $L^p(\R;X)$ for some interval $I$. By a scaling argument one sees that this implies the boundedness of the Riesz projection, and thus the Hilbert transform, and thus the $\UMD$ property. In Proposition \ref{prop:RubioBGT} we have also seen that cotype $q$ is necessary for \eqref{eq:RubioX}. Moreover, we expect that Fourier type properties can be derived in a similar way as in \cite[Section 2.4]{DengLorVer}.

It is interesting to know whether the decay condition in Theorem \ref{thm:mainthm} can be removed. A technique to do so is used in \cite{GW03, Hyt04}.
\begin{problem}
Can the decay condition be removed in Theorem \ref{thm:mainthm}? In other words, is it  enough to assume that the range of $m$ is $\mathcal{R}$-bounded and $\sup_{J\in \Delta }\|m\|_{\dot{V}^{s}(J;\mathcal{L}(X,Y))}<\infty$ to obtain the boundedness of $T_m:L^p(\R;X)\to L^p(\R,Y)$?
\end{problem}

In the scalar case, some higher dimensional analogues of \cite{coifman} have been obtained by \cite{Krol, Lac07, Xu}. In the vector-valued setting, higher dimensional analogues of Theorem \ref{5.8} can be found in \cite{HHN, SW07} and \cite[Theorem 8.3.19]{HNVW2}.
\begin{problem}
What are the analogues in $\R^d$ of the multiplier results of Theorems \ref{thm:boundedsupp}, \ref{thm:mainthm}, \ref{thm:Besov} and \ref{thm:TL}?
\end{problem}

\bigskip

{\em Use of LLMs:}
The authors used ChatGPT and Gemini for proof checking and the generation of some of the ideas.

\end{document}